\documentclass[12pt]{amsart}
\usepackage{amsmath}
\usepackage{amsxtra}
\usepackage{fancyhdr}
\usepackage{amscd}
\usepackage{amsthm}
\usepackage{amsfonts}
\usepackage{amssymb}
\usepackage{mathrsfs}
\usepackage{eucal} 
\usepackage[hidelinks, pagebackref]{hyperref}
\hypersetup{
    colorlinks,
    citecolor=black,
    filecolor=black,
    linkcolor=black,
    urlcolor=black
}
\usepackage{cite}
\usepackage{microtype}
\usepackage{url} 

\input xy
\xyoption{all}

\newtheorem{thm}{Theorem}
\newtheorem*{thm*}{Theorem}
\newtheorem{cor}[thm]{Corollary}
\newtheorem*{cor*}{Corollary}
\newtheorem{lem}[thm]{Lemma}
\newtheorem{prop}[thm]{Proposition}
\newtheorem*{prop*}{Proposition}

\newtheorem{conj}[thm]{Conjecture}

\theoremstyle{remark}
\newtheorem{rem}[thm]{Remark}

\newtheorem{question}[thm]{Question}

\theoremstyle{definition}
\newtheorem{defn}[thm]{Definition}
\numberwithin{thm}{section}

\newcommand\nc{\newcommand}
\nc\on{\operatorname}

\newcommand\fp{{\mathfrak p}}
\newcommand\fq{{\mathfrak q}}
\newcommand\fm{{\mathfrak m}}

\newcommand\mcl{\mathscr}
\newcommand\mbb{\mathbb}
\newcommand\mbf{\mathbf}
\newcommand\mfk{\mathfrak}
\nc\Hom{\on{Hom}}
\nc\Sections{\on{Sections}}
\nc\Sym{\on{Sym}}
\nc\Spec{\on{Spec}}
\nc\Specm{\on{Specm}}
\nc\ul{\underline}
\nc{\dfp}{\overset{\cdot}{\fp}}
\nc{\dfq}{\overset{\cdot}{\fq}}
\nc{\dfm}{\overset{\cdot}{\fm}}

\begin{document}

\title{Non-Abelian Lefschetz Hyperplane Theorems}
\author{Daniel Litt}
\begin{abstract}
Let $X$ be a smooth projective variety over the complex numbers, and let $D\subset X$ be an ample divisor.  For which spaces $Y$ is the restriction map $$r: {\Hom}(X, Y)\to {\Hom}(D, Y)$$ an isomorphism?

Using positive characteristic methods, we give a fairly exhaustive answer to this question.  An example application of our techniques is: if $\dim(X)\geq 3$, $Y$ is smooth, $\Omega^1_Y$ is nef, and $\dim(Y)< \dim(D),$ the restriction map $r$ is an isomorphism.  Taking $Y$ to be the classifying space of a finite group $BG$, the moduli space of pointed curves $\mcl{M}_{g,n}$, the moduli space of principally polarized Abelian varieties $\mcl{A}_g$, certain period domains, and various other moduli spaces, one obtains many new and classical Lefschetz hyperplane theorems.  
\end{abstract}

\maketitle



\tableofcontents

\cleardoublepage
\phantomsection
\pagenumbering{arabic}

\setcounter{page}{1}

\section{Introduction}
This paper arose from an attempt to understand the extent to which the properties of a variety $X$ are reflected by those of an ample Cartier divisor $D\subset X$.   Our main contribution is to give criteria for the morphism $$F(X)\to F(D)$$ to be an isomorphism, where $F$ is a contravariant functor which is representable in a suitable sense.  We obtain many classical and new Lefschetz-type theorems as corollaries of our main results.  We also hope that the results in this paper will be useful to those attempting to prove their own Lefschetz-type results; the account given here is meant to be a somewhat encyclopedic toolkit.

For example, let $X$ be a smooth projective variety over a field of characteristic zero, and let $D\subset X$ be an ample Cartier divisor.  As consequences of our main theorem (Theorem \ref{maintheorem2}), we obtain:
\begin{itemize}
\item $\pi_1^{\text{\'et}}(D)\to \pi_1^{\text{\'et}}(X)$ is surjective if $\dim(X)\geq 2$ and an isomorphism if $\dim(X)\geq 3$ (see \cite[Th\'eor\`eme XI.3.10]{SGA2}, and Theorem \ref{etalepi1-lefschetz} here);
\item Let $Y$ be a smooth projective curve of genus at least $1$.  Then $${\Hom}(X, Y)\to {\Hom}(D, Y)$$ is an isomorphism if $\dim(X)\geq 3$ (see Theorem \ref{nefcotangentbundleextension} and Theorem \ref{sourceofnefexamples});
\item Any family of smooth curves of genus $g\geq 2$ over $D$ extends uniquely to $X$, as long as $\dim(X)\geq 3g-1$ (see Corollary \ref{modulispaceapplications}(1));
\item Any principally polarized Abelian scheme over $D$ of relative dimension $g$ extends uniquely to $X$, as long as $\dim(X)\geq \dim(\mcl{A}_g)+2$ (see Corollary \ref{modulispaceapplications}(3));
\item Any polarized family of Calabi-Yau varieties over $D$ extends uniquely to $X$, as long as $\dim(X)\geq h^{n-1,1}+2$ (see Corollary \ref{modulispaceapplications}(4));
\item Let $f: Y\to X$ be a smooth relative curve of genus $g\geq 2$, and let $f_D: Y_D\to D$ be its base change to $D$.  If $\dim(X)\geq 3$, the restriction map $${\on{Sections}}(f)\to {\on{Sections}}(f_D)$$ is an isomorphism (see Corollary \ref{sectionswithglobalgeneration});
\item Let $f: Y\to X$ be an Abelian scheme of relative dimension $g$, and let $f_D: Y_D\to D$ be its base change to $D$.  If $\dim(X)\geq g+2$, the restriction map $${\on{Sections}}(f)\to {\on{Sections}}(f_D)$$ is an isomorphism (see Corollary \ref{sectionswithglobalgeneration});
\item If $\mcl{H}$ is a polarized variation of Hodge structure over $D$, induced by a period map $D\to Y$ with $Y$ quasi-projective, then $\mcl{H}$ extends uniquely to $X$ if $\dim(D)>\dim(Y)$, and if
\begin{itemize} 
\item $Y$ is compact, or 
\item  $\mcl{H}$ is of weight one, or 
\item $\mcl{H}$ is of K3 type 
\end{itemize} (see Theorem \ref{shimuralefschetz}).
\end{itemize}
We also give versions of many of these theorems in positive characteristic.

Before diving into these applications, however, we will discuss the new perspective on classical Lefschetz theorems which motivates this work.

\subsection{Classical Lefschetz theorems}\label{classicalthms}

Recall Grothendieck's Lefschetz theorem for Picard groups.
\begin{thm}[Grothendieck-Lefschetz theorem for Picard groups {\cite[Th\'eor\`eme XI.3.18]{SGA2}}]
Let $X$ be a smooth projective variety over a field of characteristic zero.  Let $D\subset X$ be an ample divisor.  Then the restriction map $\on{Pic}(X)\to \on{Pic}(D)$
\begin{itemize}
\item is injective if $\dim(X)\geq 3$ and
\item is surjective if $\dim(X)\geq 4$.
\end{itemize} 
\end{thm}
One may reinterpret this theorem as the statement that the restriction morphism $${\Hom}(X, B\mathbb{G}_m)\to {\Hom}(D, B\mathbb{G}_m),$$
is fully faithful (resp.~an equivalence).  

Similarly, recall the Lefschetz hyperplane theorem for the \'etale fundamental group.
\begin{thm}[Grothendieck-Lefschetz theorem for $\pi_1^{\text{\'et}}$ {\cite[Th\'eor\'eme X.3.10]{SGA2}}]
Let $X$ be a smooth projective variety over a field, and suppose $\dim(X)\geq 3$.  Let $D\subset X$ be an ample divisor.  Then the natural map$$\pi_1^{\text{\'et}}(D)\to \pi_1^{\text{\'et}}(X)$$ is an isomorphism (for any choice of geometric base-point in $D$).
\end{thm}
This theorem may be reinterpreted as giving conditions on $k$-varieties $X, D$ (with $D\subset X$ an ample divisor) such that the restriction map $${\Hom}(X, BG)\to {\Hom}(D, BG)$$ is fully faithful (resp.~an equivalence), where $G$ runs over all finite \'etale $k$-group schemes.

Perhaps more classically, we can rephrase the Lefschetz hyperplane theorem for singular cohomology.
\begin{thm}[Lefschetz Hyperplane Theorem]\label{lefschetz1}
Let $X$ be a smooth complex projective variety of dimension $n$, and $D\subset X$ an ample divisor.  Then the restriction map $$H^i(X, \mbb{Z})\to H^i(D, \mbb{Z})$$ is an isomorphism for $i<n-1$, and an injection for $i=n-1$.  
\end{thm}
This theorem can be reinterpreted as saying that the restriction map $$\underline{\Hom}(X, K(\mbb{Z}, n-1))\to \underline{\Hom}(D, K(\mbb{Z}, n-1))$$ is an injection on connected components, and a homotopy equivalence when restricted to each connected component of the domain.  Here $\underline{\Hom}$ denotes the space of continuous maps, with the compact-open topology.  While this interpretation seems intrinsically topological, one may make sense of it in algebraic geometry, giving a reinterpretation of the Lefschetz hyperplane theorem for e.g. \'etale cohomology.  We do not do so in this paper.

These examples should suggest a general question, of which all of these classical theorems are special cases.  Namely, fix a $k$-variety $X$ and an ample divisor $D\subset X$.  
\begin{question}\label{bigquestion}
For which spaces (schemes, stacks, etc.) $Y$ is the restriction morphism $${\Hom}(X, Y)\to {\Hom}(D, Y)$$ an equivalence?
\end{question}
We will also consider a more general question, which should be thought of as analogous to the Lefschetz hyperplane theorem for cohomology with coefficients in a local system:
\begin{question}\label{bigquestion2}
Let $Y'$ be a space and $f: Y'\to X$ a map.  Consider the diagram
$$\xymatrix{
Y_D'\ar@{^(->}[r] \ar[d]^{f_D}& Y'\ar[d]^f\\
D\ar@{^(->}[r] & X
}$$
where $Y_D'=Y'\times_X D$.  When is the restriction map $${\on{Sections}}(f)\to {\on{Sections}}(f_D)$$ an isomorphism?
\end{question}
Question \ref{bigquestion} is the case where $Y'=X\times Y$, and $f: Y'\to X$ is the projection to the first coordinate.

This paper will focus on the case where $Y$ is a smooth scheme or  Deligne-Mumford stack.  This level of generality suffices for many interesting applications.  We were unable to find a general theorem that encompassed all known Lefschetz Hyperplane theorems; for example, our methods will not prove the Lefschetz hyperplane theorem for $$H^i(X_{\text{\'et}}, \mbb{Z}/\ell\mbb{Z})\to H^i(D_{\text{\'et}}, \mbb{Z}/\ell\mbb{Z})$$ with $i>1$.  On the other hand, many of our applications seem to be new and unexpected.
\subsection{Statements}
The main result obtained in this paper is Theorem \ref{maintheorem2} below; more statements and proofs are in Chapter \ref{applications}.  These statements rely on a certain notion of positivity for vector bundles, which we define now.
\begin{defn}[Nefness for vector bundles]
Let $Y$ be a Deligne-Mumford stack, and let $\mcl{E}$ be a vector bundle on $Y$.  We say that $\mcl{E}$ is nef if for each morphism $f: C\to Y$, where $C$ is a smooth projective curve, and each surjection $$f^*\mcl{E}\to \mcl{L}\to 0$$ where $\mcl{L}$ is a line bundle on $C$, we have $$\deg_C(\mcl{L})>0.$$
\end{defn}
\begin{rem}
Note that in the above definition we do not require that $Y$ be proper; for example, any vector bundle on an affine scheme is nef.
\end{rem}
As a concrete consequence of our main theorem, we deduce
\begin{thm}\label{maintheorem}
Let $k$ be a field of characteristic zero, and $X$ a smooth projective variety over $k$.  Let $D\subset X$ be an ample Cartier divisor, and let $Y$ be a smooth Deligne-Mumford stack over $k$ with $\Omega^1_Y$ nef.  Then $${\Hom}(X, Y)\to {\Hom}(D, Y)$$ is
\begin{itemize}
\item  fully faithful if $\dim(Y) = \dim(D)$, and $\dim(X)\geq 2$
\item an equivalence if $\dim(Y)<\dim(D)$ and $\dim(X)\geq 3$.
\end{itemize}
\end{thm}
See Theorem \ref{nefcotangentbundleextension} for details.  

This theorem is a consequence of a far more general theorem, which works in arbitrary characteristic.  To state it, we recall the notion of f-amplitude of a vector bundle, from \cite{arapura-f-amplitude}.
\begin{defn}[f-amplitude]
Let $\mcl{E}$ be a vector bundle on a $k$-scheme $X$.  Then if $\on{char}(k)=p>0$, the f-amplitude of $\mcl{E}$, denoted $\phi(\mcl{E})$, is the least integer $i_0$ such that $$H^i(X, \mcl{F}\otimes \mcl{E}^{(p^k)})=0 \text{ for } k\gg 0,$$ for all coherent sheaves $\mcl{F}$ on $X$ and $i> i_0.$  Here $$\mcl{E}^{(p^k)}:=(\on{Frob}_{p^k})^*\mcl{E}$$ where $\on{Frob}_{p^k}$ is the $k$-th power of the absolute Frobenius morphism.  If $\on{char}(k)=0$, $\phi(\mcl{E})$ is defined to be the the infimum of $$\on{max}_{\mfk{q}\in A} \phi(\mcl{E}_\mfk{q}),$$ where $A$ is a finite-type $\mbb{Z}$-scheme, $(\mcl{X}, \mcl{E})$ is a model of $(X, \mcl{E})$ over $A$, and $\mfk{q}$ ranges over all closed points of $A$.  
\end{defn}
Colloquially, in characteristic zero, $\phi(\mcl{E})$ is the f-amplitude of $\mcl{E}_\mfk{p}$ on the special fibers $X_\mfk{p}$ of a well-chosen ``spreading out" of $(X, \mcl{E})$.
\begin{thm}\label{maintheorem2}
Let $k$ be a field and $X$ a smooth projective variety over $k$, with $\dim(X)\geq 3$.  Let $D\subset X$ be an ample Cartier divisor, and let $Y$ be a smooth Deligne-Mumford stack over $k$.  Let $f: D\to Y$ be a morphism.  Suppose that either $\on{char}(k)=0$ or that $k$ is perfect of characteristic $p>0$, and $X$ lifts to $W_2(k)$.  If $$\phi(f^*\Omega^1_Y\otimes \mcl{N}_{D/X})<\dim(D)$$ there is at most one extension of $f$ to $X$; if $$\phi(f^*\Omega^1_Y\otimes \mcl{N}_{D/X})<\dim(D)-1$$ and $Y$ satisfies 
\begin{enumerate}
\item $\dim(Y)<\dim(D)$, or
\item $\dim(\on{im}(f))\leq\dim(D)-2$ and $Y$ has quasiprojective coarse space, or
\item $Y$ is proper and admits a model which is finite type over $\mbb{Z}$ and whose geometric fibers contain no rational curves.
\end{enumerate}   
then $f$ extends (uniquely) to a map $X\to Y$.  
\end{thm}
See Theorem \ref{charzeroextensionthmdmstacks} for details.  We also give a version of this result for varieties over a field $k$ of positive characteristic which do \emph{not} lift to $W_2(k)$; loosely speaking, this theorem says that a map $f: D\to Y$ extends uniquely to $X$ after composition with a suitable power of Frobenius.  See Theorem \ref{frobeniusvanishing} for details. 

We also give relative versions of these theorems:
\begin{thm}\label{relativemainthm}
Let $k$ be a field of characteristic zero and $X$ a smooth projective variety over $k$.  Let $D\subset X$ be an ample Cartier divisor, and let $f: Y'\to X$ be a smooth proper morphism of schemes, and let $f_D: Y'_D\to D$ be its base change to $D$.  Suppose $\Omega^1_{Y'/X}$ is a nef vector bundle on $Y'$.  Then the restriction map $${\on{Sections}}(f)\to {\on{Sections}}(f_D)$$ is 
\begin{itemize}
\item injective if $\text{rel.~dim.}(f)= \dim(D)$ and $\dim(X)\geq 2$, 
\item an isomorphism if $\text{rel.~dim.}(f)< \dim(D)$ and $\dim(X)\geq 3$.
\end{itemize}
\end{thm}
See Theorem \ref{nefsectionsextend} and the surrounding remarks for details.
\begin{rem}
Observe that in Theorems \ref{maintheorem} and \ref{maintheorem2}, there is no properness assumption on $Y$, whereas Theorem \ref{relativemainthm} requires such an assumption.
\end{rem}
We also give analogous results in positive characteristic.  See Chapter \ref{deformation theory} for details.

In Chapter \ref{applications} we will give many applications of these results, typically by checking that various targets $Y$ have nef cotangent bundle (so as to apply Theorem \ref{maintheorem}), and that various families $Y'\to X$ have nef relative cotangent bundle (so as to apply Theorem \ref{relativemainthm}).  Some examples of smooth Deligne-Mumford stacks with nef cotangent bundle include:
\begin{itemize}
\item Smooth projective curves of genus at least $1$;
\item Abelian varieties;
\item Varieties $Y$ such that $\on{Sym}^n(\Omega^1_Y)$ is globally generated for some $n$;
\item The moduli space $\mcl{M}_{g,n}$ of $n$-pointed genus $g$ curves, with $g\geq 2$;
\item The moduli space $\mcl{A}_g$ of principally polarized Abelian varieties, over a field of characteristic zero;
\item Compact orbifolds whose universal cover is a bounded domain in $\mathbb{C}^n$ (e.g. various Shimura varieties and period domains);
\item Moduli spaces of polarized Calabi-Yau or hyperk\"ahler varieties. 
\end{itemize}
We will also show that if $Y\to X$ is a smooth morphism with relatively globally generated relative cotangent bundle, over a field of characteristic zero, the relative cotangent bundle is in fact nef.  This includes families of curves of genus $g\geq 1$, and families of Abelian varieties.

We  give  more elementary theorems along these lines, which we will refer to as ``Noether-Lefschetz-type theorems."  Recall the usual Noether-Lefschetz theorem:
\begin{thm}[Noether-Lefschetz]
Let $X$ be a smooth projective threefold and $\mcl{L}$ an ample line bundle on $X$.  Then for $n\gg 0$, and $D\in |\mathscr{L}^{\otimes n}|$ very general, the restriction map $$\on{Pic}(X)\to \on{Pic}(D)$$ is an isomorphism.
\end{thm}
For example, if $X=\mathbb{P}^3$ and $\mcl{L}=\mathscr{O}(1)$, one may take $n\geq 4$.  As usual, we give an analogue of this theorem replacing $\on{Pic}$ with $\on{Sections}(f)$ for certain maps $f$.
\begin{thm*}
Let $k$ be a field, and $X$ a smooth projective $k$-scheme of dimension $m\geq 3$.  Let $f: A\to X$ be an Abelian scheme, and let $\mcl{O}_X(1)$ be an ample line bundle on $X$.  Then for $n\gg 0$ (depending  on $A, X, \mcl{L}$), and any element $D\in |\mcl{L}^{\otimes n}|$, the restriction map $$\on{Sections}(f)\to \on{Sections}(f_D)$$ is an isomorphism.
\end{thm*}
See Theorem \ref{noether-lefschetz-abelian-schemes} for details.
\subsection{Strategy of proof}  
Following \cite{SGA2}, the main idea of the proof of these results is to factor the inclusion $$D\hookrightarrow X$$ into a series of maps.  Namely, let $\widehat D$ be the formal scheme obtained by completing $X$ at $D$.  Then the inclusion above factors as $$D\hookrightarrow \widehat D\hookrightarrow U\hookrightarrow X$$ where $U$ is a Zariski-open neighborhood of $D$.  Applying the functor ${\Hom}(-, Y)$ we obtain restriction maps $${\Hom}(X, Y)\to\varinjlim_{D\subset U}{\Hom}(U, Y)\to {\Hom}(\widehat D, Y)\to {\Hom}(D, Y).$$
We study each of the maps above separately.  
\subsubsection{Extension}
The question of studying the map $${\Hom}(X, Y)\to\varinjlim_{D\subset U}{\Hom}(U, Y)$$ boils down to: when can a morphism from a Zariski-open subset of a variety be extended to the entire variety?  We refer to this question as the \emph{extension} question, and address it in Chapter \ref{extension}.  We identify two situations where such maps extend:  first, when the target $Y$ satisfies $$\dim(Y)\leq \dim(X)-2,$$ and second, when $X$ is smooth and $Y$ contains no rational curves.  Our main results are:
\begin{cor*}
Let $X$ be a normal projective $k$-variety, and $Y$ a $k$-variety such that 
\begin{enumerate}
\item $X$ is locally $\mbb{Q}$-factorial, or
\item $Y$ is quasi-projective.
\end{enumerate}
Let $D\subset X$ be an ample divisor and $U\subset X$ a Zariski-open subset containing $D$.  Then if $\dim(Y)\leq \dim(X)-2$, any map $U\to Y$ extends uniquely to a map $X\to Y$.
\end{cor*}
and 
\begin{prop*}
Let $X$ be a smooth variety over a field $k$, and let $f:Y \to X$ be a proper morphism.  If the geometric fibers of $f$ contains no rational curves, any rational section of $f$, denoted $s: X\dashrightarrow Y$, extends uniquely to a regular section $X\to Y$.
\end{prop*}
See Corollary \ref{nonproperextension} and Proposition \ref{noratlcurvesextension} for details.  We also give versions of these results in somewhat more general settings, e.g.~over general base schemes.  We also include some results on maps with small image, i.e.~Corollary \ref{smallimageextension}.
\subsubsection{Algebraization}
The properties of the morphism $$\varinjlim_{D\subset U}{\Hom}(U, Y)\to {\Hom}(\widehat D, Y)$$ follow from what we term an \emph{algebraization} question:  when does a morphism of formal schemes $\widehat D\to Y$ extend to a Zariski-open neighborhood of $D$?  In Chapter \ref{algebraization}, we study a substantially more general question, following methods from \cite[Expos\'e XII]{SGA2}.  Namely, let $f: Y\to X$ be a morphism, and let $\widehat{Y_D}$ be the formal scheme obtained by completing $Y$ at $f^{-1}(D)$; we study the restriction map $$\on{QCoh}(Y)\to \on{QCoh}(\widehat{Y_D}),$$ and from $$\on{Perf}(Y)\to \on{Perf}(\widehat{Y_D}).$$  The original question of extending a morphism $g: \widehat D\to Y'$ follows by considering the structure sheaf of the graph of $g$.  Our main results are:
\begin{cor*}
Let $k$ be a field, $g: Y\to X$  a morphism of finite type and $X$ a projective normal $k$-variety, with $D\subset X$ an ample Cartier divisor.  Let $\widehat g: \widehat{Y_D}\to \widehat D$ be the completion of $g$ at $D$.  Suppose that the dualizing sheaf $\omega_{X/k}$ has coherent cohomology supported in degrees $[-n, -m]$, with $m\geq 2$.  Then any section to $\widehat g$ extends to a Zariski-open neighborhood of $D$
\end{cor*}
and
\begin{cor*}
Let $k$ be a field, $g: Y\to X$ be a morphism of finite type and $X$ a projective normal $k$-variety, with $D\subset X$ an ample Cartier divisor.  Let $\widehat g: \widehat{Y_D}\to \widehat D$ be the completion of $g$ at $D$.  Suppose that the dualizing sheaf $\omega_{X/k}$ has coherent cohomology supported in degrees $[-n, -m]$, with $m\geq 2$.  Then two sections to $g$ agreeing on $\widehat D$ are equal.
\end{cor*}
See Corollaries \ref{sectionalgebraization} and \ref{sectionsareequal}, and the surrounding remarks, for details.
\subsubsection{Deformation Theory}
This section, contained in Chapter \ref{deformation theory}, is the heart of the argument.  We consider the morphism $${\Hom}(\widehat D, Y)\to {\Hom}(D, Y).$$  The properties of this morphism are questions of pure \emph{deformation theory}, which we briefly review.  

Suppose $Y$ is smooth, and let $\mcl{I}_D$ be the ideal sheaf of $D$; let $D_n$ be the Cartier divisor defined by $\mcl{I}_D^n$.  Then obstructions to deforming a map $f: D\to Y$ to a map $f_2: D_2\to Y$ lies in $$\on{Ext}^1(\mcl{N}_{D/X}, f^*T_Y),$$ and if this obstruction vanishes, such deformations are a torsor for $$\Hom(\mcl{N}_{D/X}, f^*T_Y).$$  Thus the existence of deformations are implied by the vanishing of $\on{Ext}^1(\mcl{N}_{D/X}, f^*T_Y),$ and their uniqueness is implied by the vanishing of $\Hom(\mcl{N}_{D/X}, f^*T_Y).$    For deformations to $D_n$ with $n>2$, see Theorem \ref{maindefthm}.

In favorable situations, (e.g. if $\Omega^1_Y$ is nef and $D$ is smooth, and everything is taking place over a field of characteristic zero), one may prove this vanishing via the Le Potier vanishing theorem.  The main work of Chapter \ref{deformation theory} is in dealing with the case where $D$ is not smooth, and the case of positive characteristic.  Indeed, even in characteristic zero, if $D$ is not smooth our arguments go through characteristic $p>0$.  These ideas owe a great deal to suggestions of Bhargav Bhatt, and to the work of Donu Arapura.

The main result over an arbitrary field $L$ of positive characteristic is:
\begin{thm*}
Let $X$ be variety over a  field $L$ of characteristic $p$, and let $D\subset X$ be a Cartier divisor whose dualizing complex $K_D$ is supported in degrees $[-\dim(D), -r]$, and whose normal bundle is ample.  Suppose $\dim(X)\geq 3$.  Let $\widehat D$ be the formal scheme obtained by completing $X$ at $D$.  Let $f: D\to Y$ be a morphism, with $Y$ a smooth $k$-variety.  Suppose that $$\phi(f^*\Omega^1_Y\otimes \mcl{N}_{D/X})<r-1.$$  Let $\widetilde f^{(p^k)}=F_{Y/L}^k\circ f: D\to Y^{(p^k)}$.  Then for $k\gg0, \widetilde f^{(p^k)}$ extends uniquely to a morphism $\widehat D\to Y^{(p^k)}$.
\end{thm*}
See Theorem \ref{frobeniusvanishing} for details and notation.  If $L$ is perfect and $X$ lifts to $W_2(L)$, we deduce:
\begin{thm}
Let $L$ be a perfect field of characteristic $p>0$.  Let $X$ be a smooth $L$-variety and $D\subset X$ an ample Cartier divisor, with $\dim(X)\geq 3$, such that $X$ lifts to $W_2(L)$, and such that $\dim(X)<p$.  Let $Y$ be a smooth $k$-variety, and let $f: D\to Y$ be a morphism. Suppose that  $$\phi(\mcl{N}_{D/X}\otimes f^*\Omega^1_Y)<\dim(D)-1.$$ Suppose further that  
\begin{enumerate}
\item $\dim(Y)<\dim(D)$, or
\item $Y_{\bar L}$ contains no rational curves.
\end{enumerate}   
Then $f$ extends uniquely to a morphism $X\to Y$.  
\end{thm} 
For arbitrarily singular $D$, we are unable to directly deform that map $f: D\to Y$ to a map $\widehat D\to Y$; rather we show that $F_{Y/L}^k\circ f$ extends to a map $X\to Y$, and then use the smoothness of $X$ to deduce that in fact it was unnecessary to compose with Frobenius.  See Theorem \ref{charpliftextensionthm} for details.  We use this result to deduce our main results in chracteristic zero.
\subsection{Remarks}
\subsubsection{Past Work}
There is a large body of literature on Questions \ref{bigquestion} and \ref{bigquestion2}, approaching it from a rather different point of view.  Sommese \cite{Sommese:1976aa} was interested in characterizing smooth projective varieties which cannot be ample divisors.  He proves:
\begin{thm}[Sommese, {\cite[Theorem 3.1]{extending-morphisms}}, \cite{Sommese:1976aa}]
Let $D$ be a smooth ample divisor on a smooth projective variety $X$ over $\mbb{C}$, with $\dim(X)\geq 4$.  Let $p: D\to Z$ be a surjective morphism.  Then if $$\dim(D)-\dim(Z)\geq 2,$$ $p$ extends to a surjective morphism $X\to Z$.
\end{thm}
Sommese views this as a \emph{negative} result, obstructing the possibility of a variety appearing as an ample divisor in another variety; in contrast, our philosophy in Section \ref{classicalthms} suggests that one should view this as a positive result, via careful choice of $Z$.  This beautiful theorem sparked an industry (see \cite{extending-morphisms} and the references therein, e.g. \cite{Silva:1977aa}) centered on comparing maps out of a variety (and in particular, contractions) to maps out of an ample divisor.  

Our goal in this paper is to conclusively answer some of the questions raised by this body of work--namely, to do away with assumptions about $D$, which we work rather hard to avoid in Chapter \ref{deformation theory}.  Sommese's method of proof also limits one to working in characteristic zero and $\dim(X)\geq 4$, because it relies on the Grothendieck-Lefschetz theorem for Picard groups, which fails in positive characteristic.  

Our main discovery is that Sommese's hypothesis on the dimension of $Y$ is far too restrictive (especially in applications).  While we do recover Sommese's results as applications of our main theorem, we obviate the requirement that $$\dim(D)-\dim(Z)\geq 2,$$ in e.g. Theorem \ref{maintheorem2}.   We also drop this assumption to achieve the (sharp) result in the case $Z$ has nef cotangent bundle (Theorem \ref{maintheorem}).  Furthermore, Sommese's result requires the target be a projective variety, and that the map $p$ be surjective, which is insufficient for most of our applications (e.g. where the target is the non-proper Deligne-Mumford stack $\mcl{M}_g$).   Finally, Sommese's methods do not allow one to study Question \ref{bigquestion2}, that is, the question of extending sections to a smooth morphism.

In Theorem \ref{smalltargetsextension}, we improve this result to 
\begin{thm*}
Let $k$ be a field of characteristic zero.  Let $X$ be a smooth projective $k$-variety with $\dim(X)\geq 3$, and let $D\subset X$ be an ample Cartier divisor.  Let $Y$ be a smooth Deligne-Mumford stack over $k$ and let $f: D\to Y$ be a morphism.  Suppose that there exists a scheme $Y'$ and a finite surjective \'etale morphism $Y'\to Y$, and that  $\dim(Y)<\dim(D)-1$.  Then $f$ extends uniquely to a morphism $X\to Y$. 
\end{thm*}
Observe that we allow $D$ to have arbitrary singularities.

It is also worth mentioning another result of Beltrametti and Sommese which one may view as a prototype of our Theorem \ref{maintheorem}.  In \cite[Theorem 5.2.3]{C.:2011aa}, Sommese and Beltrametti show
\begin{thm}
Let $D$ be a normal, ample Cartier divisor in a normal variety $X$, and let $Z$ be a variety admitting a finite-to-one map to an Abelian variety, and so that $$\dim(D)-\dim(Z)\geq 1.$$ Then a surjective morphism $p: D \to Z$ extends to a morphism $X\to Z$.
\end{thm}
\begin{rem}
The condition that $Z$ admits a finite-to-one map to an Abelian variety is close to nefness; it is equivalent to the condition that $$\on{coker}(\Gamma(Z, \Omega^1_Z)\otimes \mathscr{O}_Z\to \Omega^1_Z)$$ be torsion.  The requirement that $Z$ admit an unramified map to an Abelian variety is equivalent to global generation of $\Omega^1_Z$, which of course implies nefness.
\end{rem}
Again, Sommese and Beltrametti work over the complex numbers.  Unfortunately, this theorem does not suffice for our applications once again; even if one were to extend it stacks, the targets $Z$ in many of our applications do not admit finite maps to Abelian varieties (e.g. $\mathscr{M}_g$ does not).

Finally, we identify some conditions in which the dimension of the target is irrelevant to the conclusion, e.g.~in Theorem \ref{f-semipositive-lefschetz-thm}, or in Theomem \ref{maintheorem2}(3).  We expect this to be important in applications.

The other main prototype of this work is Grothendieck and Raynaud's masterpiece \cite{SGA2}, in which Lefschetz Hyperplane theorems are proven for the Picard Group, \'etale fundamental group, etc.  Our strategy of proof broadly follows the strategy in that work, though it is logically independent.

This work is also inspired to a large extent by \cite{bhatt-dejong}, which uses positive characteristic methods to prove a local Lefschetz theorem, conjectured by K\'ollar.

\subsubsection{Predictions of the Hodge Conjecture}
The Hodge conjecture makes predictions similar to what we prove here.  As usual, let $X$ be a projective variety, and let $D\subset X$ be a \emph{smooth} ample divisor.  Let $Y$ be a smooth projective variety.  Let $f: D\to Y$ be a morphism and $$[\Gamma_f]\in H^{2\dim(Y)}(D\times Y, \mathbb{Z})=\bigoplus_{i+j=2\dim(Y)} H^i(D, \mbb{Z})\otimes H^j(Y, \mbb{Z})$$ the fundamental class of the graph of $f$; in fact this class has only one non-zero component, in $$H^{\dim(Y)}(D, \mbb{Z})\otimes H^{\dim(Y)}(Y, \mbb{Z}).$$  If $\dim(Y)<\dim(D)$, this class comes from a class in $$H^{\dim(Y)}(X, \mbb{Z})\otimes H^{\dim(Y)}(Y, \mbb{Z})\subset H^{2\dim(Y)}(X\times Y, \mbb{Z})$$ by the Lefschetz hyperplane theorem for singular cohomology.  Thus on the Hodge conjecture, some multiple of this class is represented by a subvariety of $X\times Y$, which at least numerically looks like the graph of a map $X\to Y$.  We have shown that in many cases, the map $D\to Y$ does indeed extend to a map $X\to Y$.

A similar argument shows that in general, correspondences of suitable dimension between $D$ and $Y$ extend, after taking some multiple and modifying by homologically trivial correspondences, to a correspondence between $X$ and $Y$.  The generality in which we work in Chapter \ref{algebraization} allows one to algebraize correspondences, and extending correspondences across points is doable, but deformation-theoretic study of correspondences is difficult, and we were unable to write Chapter \ref{deformation theory} in this generality.  

There are other conjectures making similar projections.  Hartshorne asks in \cite{hartshorne-chow}:
\begin{question}
Let $D\subset X$ be an ample divisor.  Is the restriction map $$CH^i(X)\to CH^i(D)$$ an isomorphism for $i<\dim(D)/2$?  
\end{question}
This question is open even for hypersurfaces in $\mathbb{P}^6$.  An analogous question for $k$-ample divisors would also suggest a version of our main theorem for correspondences.
\subsubsection{Acknowledgments}
This work owes a great debt to my advisor, Ravi Vakil, whose optimism and mathematical curiosity are unmatched.  The ideas in Chapter \ref{deformation theory} benefited enormously from suggestions of Bhargav Bhatt; the inspiring paper \cite{arapura-f-amplitude} was indispensible, as was the encouragement of its author, Donu Arapura.  Without the opportunity to learn from Brian Conrad's technical strength (and love of the formal GAGA theorem), I would never even have gotten started.  I am also extremely grateful for conversations with Rebecca Bellovin, Jeremy Booher, Johan de Jong, Soren Galatius, Benedict Gross, Zhiyuan Li, Cary Malkewicz, Mircea Musta\c{t}\u{a}, John Pardon, Niccolo Ronchetti, Burt Totaro, Arnav Tripathy, Akshay Venkatesh, and Zhiwei Yun.

\section{Algebraization of coherent sheaves}\label{algebraization}
In this section we consider the question of extending a sheaf from a formal subscheme of a scheme to a neighborhood thereof.  Suppose $X$ is a normal projective variety, $D\subset X$ an ample Cartier divisor, and $f: Y\to X$ a morphism.  Let $\widehat D$, (resp.~$\widehat{Y_D}$) denote the formal scheme obtained by completing $X$ at $D$ (resp.~completing $Y$ at $f^{-1}(D)$).  Then we will compare coherent sheaves on $\widehat Y_D$ to sheaves on a neighborhood $U\subset Y$ of $Y_D$.  In its broad outlines, this discussion will follow the proofs of the formal GAGA theorem and Grothendieck's existence theorem, though the details will be somewhat different.  Taking $Y=X$ will imply essentially all of the results of \cite[Expos\'e XII]{SGA2}, and our proofs will follow the general structure of the arguments there, though our arguments will be slightly cleaner due to our use of perfect complexes and Grothendieck duality.  One should also compare these results to those of \cite{raynaud}.
\subsection{The comparison theorem}  
We first observe a corollary of Serre vanishing, namely that anti-ample line bundles kill cohomology in low degree.
\begin{lem}\label{antiample vanishing}
Let $S$ be a Noetherian scheme.  Let $f: X\to S$ be projective with dualizing complex ${\omega_{X/S}=f^!\mcl{O}_S}$ having coherent cohomology concentrated in degrees $[-n, -m]$.  Let $\mcl{F}$ be a perfect complex  on $X$ whose $f^{-1}\mcl{O}_S$ tor-amplitude is in $[-a, 0]$.  Let $\mcl{L}$ be an $S$-ample line bundle on $X$.  Then for $n\gg 0, 0\leq i< m-a$, $$\mbf{R}^if_*(\mcl{F}\otimes (\mcl{L}^\vee)^{\otimes n})=0.$$
\end{lem}
\begin{proof}
By Grothendieck duality, we have 
\begin{align*}
\mbf{R}\underline{\on{Hom}}(\mbf{R}f_*(\mcl{F}\otimes (\mcl{L}^\vee)^{\otimes n}), \mcl{O}_S) &= \mbf{R}f_*(\mbf{R}\underline{\on{Hom}}(\mcl{F}\otimes (\mcl{L}^\vee)^{\otimes n}, \omega_{X/S}))\\
&=\mbf{R}f_*(\mcl{F}^\vee \otimes^{\mathbf{L}} \omega_{X/S}\otimes \mcl{L}^{\otimes n}).
\end{align*}
Now $\mcl{F}^\vee \otimes^{\mathbf{L}} \omega_{X/S}$ has coherent cohomology concentrated in degrees ${[-n, -m+a]}$.  There is a spectral sequence $$\mbf{R}^if_*(\mcl{H}^j(\mcl{F}^\vee\otimes^{\mathbf{L}} \omega_{X/S})\otimes \mcl{L}^{\otimes n})\to \mbf{R}^{i+j}f_* (\mcl{F}^\vee \otimes^{\mathbf{L}} \omega_{X/S}\otimes \mcl{L}^{\otimes n}).$$  But for $n\gg 0$ and $i>0$, $$\mbf{R}^if_*(\mcl{H}^j(\mcl{F}^\vee\otimes^{\mathbf{L}} \omega_{X/S})\otimes \mcl{L}^{\otimes n})=0$$ by Serre vanishing.  So for $n\gg0$, $\mbf{R}f_*(\mcl{F}^\vee \otimes^{\mathbf{L}} \omega_{X/S}\otimes \mcl{L}^{\otimes n})$ and thus $\mbf{R}\underline{\on{Hom}}(\mbf{R}f_*(\mcl{F}\otimes (\mcl{L}^\vee)^{\otimes n}), \mcl{O}_S)$  have cohomology concentrated in degrees $[-n, -m+a]$.  As $\mbf{R}_f^*(\mcl{F}\otimes (\mcl{L}^\vee)^{\otimes n})$ is perfect, by e.g. \cite[Tag 0A1E, Lemma, 35.18.1]{stacks-project}, $$\mbf{R}f_*(\mcl{F}\otimes (\mcl{L}^\vee)^{\otimes n})$$ has cohomology concentrated in degrees $[m-a, n]$, as desired.
\end{proof}
\begin{lem}[Formal GAGA: Grauert comparison theorem (Compare to {\cite[Expos\'e XII, Th\'eor\`eme 2.1(ii)]{SGA2}})] \label{grauert2}
Let $f: X\to S$ and $\mcl{L}$ be as before.   Let $\mcl{F}$ be a perfect complex on $X$ with tor-amplitude $[-a, 0]$ and $D\in |\mcl{L}|$ a divisor with ideal sheaf $\mcl{I}_D\subset \mcl{O}_X$.  Let $\widehat{D}$ be the completion of $X$ at $D$ and $\widehat f: \widehat{D}\to S$ the restriction of $f$ to $\widehat{D}$; let $\widehat{\mcl{F}}$ be the restriction of $\mcl{F}$ to $\widehat{D}$.  Then for $0\leq i\leq m-a-2$ the natural map $$\mbf{R}^i\widehat{f}_*\widehat{\mcl{F}}\to \varprojlim_n \mbf{R}^if_*(\mcl{F}\otimes^{\mathbf{L}}\mcl{O}_X/\mcl{I}_D^n)$$ is an isomorphism. 
\end{lem}
\begin{proof}
First  observe that the functor $\mcl{G}\mapsto \widehat{\mcl{G}}$ is exact (using the Noetherianity of $S$).  As $\mbf{R}f_*$ and $\mbf{R}\varprojlim$ commute by \cite[Tag 07D6, Lemma 19.13.6]{stacks-project}, we have that the natural map $$\mbf{R}\widehat{f}_*\widehat{\mcl{F}}\to \mbf{R}\varprojlim \mbf{R}f_*(\mcl{F}\otimes^{\mathbf{L}} \mcl{O}_X/\mcl{I}_D^n)$$ is an isomorphism.  By the Grothendieck spectral sequence, it suffices to show that $\mbf{R}^i\varprojlim \mbf{R}^jf_* (\mcl{F}\otimes^{\mathbf{L}} \mcl{O}_X/\mcl{I}_D^n)=0$ for $i>0$ and $j\leq m-a-2$.  

Without loss of generality, $S$ is affine.  Now from the short exact sequence $$0\to \mcl{I}_D^n\to \mcl{O}_X \to \mcl{O}_X/\mcl{I}_D^n\to 0$$  one obtains a distinguished  triangle $$ \mcl{I}_D^n\otimes \mcl{F}\to \mcl{F}\to \mcl{O}_X/\mcl{I}_D^n \otimes \mcl{F}\to\mcl{I}_D^n\otimes \mcl{F}[1].$$
For $n\gg 0$, $\mcl{I}_D^n\otimes \mcl{F}$ is $(m-a-1)$-connected by Lemma \ref{antiample vanishing}.  Thus $$\mbf{R}^j\Gamma(\mcl{F})\to \mbf{R}^j\Gamma(\mcl{O}_X/\mcl{I}_D^n\otimes \mcl{F})$$ is an isomorphism for $0\leq j \leq m-a-2$.  But for $r<n$ the triangle
$$\xymatrix{
R^j\Gamma(\mcl{F}) \ar[rr]^\sim \ar[rd]^\sim & & R^j\Gamma(\mcl{O}_X/\mcl{I}^n_D\otimes \mcl{F}) \ar[ld]\\
& R^j\Gamma(\mcl{O}_X/\mcl{I}^r_D\otimes \mcl{F})&
}$$
commutes, and for $n, r\gg0$ the top and left arrows are isomorphisms.  Thus for $r, n\gg 0$ the maps $R^j\Gamma(\mcl{O}_X/\mcl{I}^n_D\otimes  \mcl{F})\to R^j\Gamma(\mcl{O}_X/\mcl{I}^r_D\otimes \mcl{F})$ are isomorphisms.

In particular, for $0\leq j \leq m-a-2$, the projective systems $(\mbf{R}^j\Gamma(\mcl{O}_X/\mcl{I}_D^n\otimes \mcl{F}))_{n\in \mcl{N}}$ satisfy the Mittag-Leffler condition, and so $\mbf{R}^i\varprojlim \mbf{R}^j\Gamma(\mcl{O}_X/\mcl{I}_D^n\otimes \mcl{F})=0$ for $i>0$ and $j\leq m-a-2$, as desired.
\end{proof}
\begin{lem}[Grauert comparison theorem for ample divisors (Compare to {\cite[Expos\'e XII Th\'eor\`eme 2.1(i)]{SGA2}})] \label{grauert1}
Let $f: X\to S$ and $\mcl{L}$ be as before.   Let $\mcl{F}$ be a perfect complex on $X$ with tor-amplitude $[-a, 0]$ and $D\in |\mcl{L}|$ a divisor with ideal sheaf $\mcl{I}_D\subset \mcl{O}_X$.  Let $\widehat{X}$ be the completion of $X$ at $D$ and $\widehat f: \widehat{X}\to S$ the restriction of $f$ to $\widehat{X}$; let $\widehat{\mcl{F}}$ be the restriction of $\mcl{F}$ to $\widehat{X}$.  Then for $0\leq i\leq m-a-2$ the natural map $$R^if_*\mcl{F}\to R^i\widehat{f}_*\widehat{\mcl{F}}$$ is an isomorphism.
\end{lem}
\begin{proof}
By Lemma \ref{grauert2}, it suffices to show that $$\mbf{R}^if_*\mcl{F}\to \mbf{R}^if_*(\mcl{F}\otimes^{\mathbf{L}} \mcl{O}_X/\mcl{I}_D^n)$$ is an isomorphism for $n\gg0$ and $i\leq m-a-2$.  But this is exactly Lemma \ref{antiample vanishing}, using the distinguished triangle  $$\mbf{R}f_*(\mcl{F}\otimes \mcl{I}_D^n)\to\mbf{R}f_*\mcl{F}\to \mbf{R}f_*(\mcl{F}\otimes^{\mbf{L}}\mcl{O}_X/\mcl{I}_D^n)\to\mbf{R}f_*(\mcl{F}\otimes \mcl{I}_D^n)[1].$$
\end{proof}
Recall the statement of formal GAGA over a general base \cite[$\on{III}_1$.4.1.5]{EGA}.
\begin{thm}[Formal GAGA {\cite[$\on{III}_1$.4.1.5]{EGA}}]\label{formalgaga}
Let $f: X\to Y$ be a proper morphism of Noetherian schemes, and let $Z\subset Y$ be a closed subscheme.  Let $\widehat f: \widehat{X_Z} \to \widehat Z$ be the associated map of formal schemes obtained by completing $Y$ at $Z$ (resp. $X$ at $f^{-1}(Z)$).  Then $$\widehat{\mbf{R}f_*\mcl{F}}\to \mbf{R}{\widehat f}_* \widehat{\mcl{F}}$$ is an isomorphism.  
\end{thm}
\begin{cor}[Formal GAGA over a divisor]\label{maincomparison}
Let $g: Y\to X$ be a proper morphism and $f: X\to S$ projective, with $D\subset X$ an $f$-ample Cartier divisor.  Suppose that $\omega_{X/S}$ has coherent cohomology concentrated in degrees $[-n, -m]$.  Let $\mcl{F}$ be a complex on $Y$ so that $\mbf{R}g_*\mcl{F}$ is perfect with tor-amplitude $[-a, 0]$.  Let $\widehat g: \widehat{Y_D}\to \widehat D$ be the completion of $g$ at $D$, and $\widehat f: \widehat D\to S$ the structure morphism.  Then if $0\leq i\leq m-a-2$, the map $$\mbf{R}^i(f\circ g)_*\mcl{F}\to \mbf{R}^i(\widehat f\circ \widehat g)_*\widehat{\mcl{F}}$$ is an isomorphism.
\end{cor}
\begin{proof}
By Lemma \ref{grauert1}, there is an isomorphism $$\mbf{R}^if_*\mbf{R}g_*\mcl{F}\to \mbf{R}^i\widehat f_*\widehat{\mbf{R}g_*\mcl{F}}$$ for $0\leq i\leq m-a-2.$  By Theorem \ref{formalgaga} the natural map $$\widehat{\mbf{R}g_*\mcl{F}}\to \mbf{R}\widehat{g}_*\widehat{\mcl{F}}$$ is an isomorphism; combining these two facts gives the claim immediately.
\end{proof}
\begin{cor}[Analogue of {\cite[$\on{III}_1$.4.5.1]{EGA}}]\label{mapcomparison}
Let $g: Y\to X$ be a proper morphism and $f: X\to S$ projective, with $D\subset Y$ an $f$-ample Cartier divisor.  Let $\mcl{F}, \mcl{G}$ be complexes on $Y$ so that $$\mbf{R}g_*\mbf{R}\Hom(\mcl{F}, \mcl{G})$$ is perfect of tor-amplitude $[-a, 0]$.  Let $\widehat g: \widehat{Y_D}\to \widehat D$ be the completion of $g$ at $D$, and $\widehat f: \widehat D\to S$ the structure morphism.  Then the map $$\on{Ext}^i(\mcl{F}, \mcl{G})\to \on{Ext}^i(\widehat{\mcl{F}}, \widehat{\mcl{G}})$$ is an isomorphism for $0\leq i\leq m-a-2.$
\end{cor}
\begin{proof}
This is immediate by applying Corollary \ref{maincomparison} to the complex $\mbf{R}\Hom(\mcl{F}, \mcl{G})$.
\end{proof}
\subsection{The existence theorem and corollaries}
Let $k$ be a field, $X$ a projective $k$-scheme and $f: Y\to X$ a morphism.  Let $D\subset X$ be an ample divisor and $Y_D=f^{-1}(D)$.  Finally, let $\widehat{Y_D},$ (resp.~$\widehat D$) be the formal scheme arising as the completion of $Y$ at $Y_D$ (resp.~the completion of $X$ at $D$.  In this section, we study the problem of algebraizing complexes and coherent sheaves on $\widehat{Y_D}$.  That is, given a sheaf or complex $\mcl{F}$ on $\widehat{Y_D}$, we will want to find a Zariski-open $U\subset Y$ containing $Y_D$, and extension of $\mcl{F}'$ to $U$, so that $\mcl{F}'|_{\widehat{Y_D}}=\mcl{F}$.  If $\mcl{F}$ has certain properties (e.g. it is a coherent sheaf, flat over $X$), we would like to arrange that $\mcl{F}'$ does as well.  Finally, we will study the extent to which such extensions $\mcl{F}'$ are unique.
\subsubsection{The Theorem}
We begin with the case where $f: Y\to X$ is projective; let $\mcl{O}(1)$ be an $f$-ample line bundle on $X$.  By \cite[$\on{III}_1$.5.2.4]{EGA}, if $\mcl{F}$ is a coherent sheaf on $\widehat{Y_D}$, there exists a surjection $$\widehat{\mcl{O}_Y(-m_1)}\otimes \widehat{f}^*\widehat{f}_*\mcl{F}(m_1)\to \mcl{F}\to 0$$ for $m\gg0$.    Applying \cite[$\on{III}_1$.5.2.4]{EGA} again, we may find a surjection $$\widehat{\mcl{O}_X(-a_1D)^{n_1}}\to \widehat{f}_*\mcl{F}(m_1)\to 0.$$  Combining these constructions, we obtain a surjection $$\widehat{\mcl{O}_Y(-m_1)}\otimes \widehat{f}^*\widehat{\mcl{O}_X(-a_1D)^{n_1}}\to \mcl{F}\to 0.$$  Applying this argument to the kernel of the map above, we find a presentation $$\widehat{\mcl{O}_Y(-m_2)}\otimes \widehat{f}^*\widehat{\mcl{O}_X(-a_2D)^{n_2}}\to \widehat{\mcl{O}_Y(-m_1)}\otimes \widehat{f}^*\widehat{\mcl{O}_X(-a_1D)^{n_1}}\to \mcl{F}\to 0;$$ we may take $m_2-m_1$ arbitrairly large.
\begin{cor}\label{existencetheorem}
Let $k$ be a field, $g: Y\to X$  a quasi-projective morphism of finite type and $X$ a projective normal $k$-variety, with $D\subset X$ an ample Cartier divisor.  Let $\widehat g: \widehat{Y_D}\to \widehat D$ be the completion of $g$ at $D$.  Suppose that the dualizing sheaf $\omega_{X/k}$ has coherent cohomology supported in degrees $[-n, -m]$, with $m\geq 2$.  Then if $\mcl{F}$ is a coherent sheaf on $\widehat{Y_D}$ with support proper over $\widehat D$, there exists a coherent sheaf $\mcl{G}$ on $Y$ so that $\mcl{F}\simeq \widehat{\mcl{G}}$.
\end{cor}
\begin{proof}
Without loss of generality, $g$ is projective and flat (by replacing $Y$ with a suitable projective bundle over $X$ in which it embeds).  Choose a resolution  $$\widehat{\mcl{O}_Y(-m_2)}\otimes \widehat{g}^*\widehat{\mcl{O}_X(-a_2D)^{n_2}}\overset{p}{\to} \widehat{\mcl{O}_Y(-m_1)}\otimes \widehat{g}^*\widehat{\mcl{O}_X(-a_1D)^{n_1}}\to \mcl{F}\to 0$$ as above, with $m_1-m_2\gg0$, so that $$\mbf{R}g_*\mbf{R}\Hom({\mcl{O}_Y(-m_2)}\otimes {g}^*{\mcl{O}_X(-a_2D)^{n_2}}, {\mcl{O}_Y(-m_1)}\otimes {g}^*{\mcl{O}_X(-a_1D)^{n_1}})$$ is locally free by Serre vanishing and the flatness of $g$.  Thus by Corollary \ref{mapcomparison}, the map $p$ algebraizes to a map $p'$; let $\mcl{G}=\on{coker}(p')$.  By the exactness of completion, $\mcl{F}\simeq \widehat{\mcl{G}}$ as desired.
\end{proof}
\begin{lem}\label{quasiprojectivecomparison}
Let $k$ be a field, $g: Y\to X$  a quasi-projective morphism of finite type and $X$ a projective normal $k$-variety, with $D\subset X$ an ample Cartier divisor.  Let $\widehat g: \widehat{Y_D}\to \widehat D$ be the completion of $g$ at $D$.  Suppose that the dualizing sheaf $\omega_{X/k}$ has coherent cohomology supported in degrees $[-n, -m]$, with $m\geq 2$.  
Then if $\mcl{G}$ is a coherent sheaf on $\widehat{Y_D}$, flat with support proper over $\widehat D$, and $\mcl{F}$ is an arbitrary coherent sheaf, the map $$\on{Ext}^i(\mcl{F}, \mcl{G})\to \on{Ext}^i(\widehat{\mcl{F}}, \widehat{\mcl{G}})$$ is an isomorphism for $0\leq i\leq m-a-2$.
\end{lem}
\begin{proof}
We may assume $g$ is projective by replacing $Y$ with a suitable projective space over $X$ in which $Y$ embeds.  Then we may choose a resolution of $\widehat{\mcl{F}}$ by a complex of the form $$\cdots\to {\mcl{O}_Y(-m_2)}\otimes {g}^*{\mcl{O}_X(-a_2D)^{n_2}}\to{\mcl{O}_Y(-m_1)}\otimes {g}^*{\mcl{O}_X(-a_1D)^{n_1}}\to \mcl{F}\to 0,$$ with $m_i\gg0$.  Thus we will have a spectral sequence $$\on{Ext}^i(\mcl{O}_Y(-m_j)\otimes g^*\mcl{O}_X(-a_j)^{n_j}, \mcl{G})=H^i(Y, \mcl{G}\otimes g^*\mcl{O}_X(a_j)^{n_j}\otimes \mcl{O}_Y(m_j))\implies \on{Ext}^{i+j}(\mcl{F}, \mcl{G})$$
But as we may choose $m_i\gg0$, the the result is immediate by Corollary \ref{maincomparison}, as $$\mbf{R}g_*(\mcl{G}\otimes g^*\mcl{O}_X(a_j)\otimes \mcl{O}_Y(m_j))$$ will be a vector bundle for $m_j\gg0$.
\end{proof}
\subsubsection{Corollaries}
Here we record some corollaries of the work above which we will use throughout this paper.
\begin{cor}\label{subschemealgebraization}
Let $k$ be a field, $g: Y\to X$  a morphism of finite type and $X$ a projective normal $k$-variety, with $D\subset X$ an ample Cartier divisor.  Let $\widehat g: \widehat{Y_D}\to \widehat D$ be the completion of $g$ at $D$.  Suppose that the dualizing sheaf $\omega_{X/k}$ has coherent cohomology supported in degrees $[-n, -m]$, with $m\geq 2$.  Then any (formal) subscheme of $\widehat{Y_D}$ extends to a Zariski-open neighborhood of $Y_D$.
\end{cor}
\begin{proof}
We first consider the case where $g$ is quasiprojective; as before, $g$ is without loss of generality a projective bundle over $X$.  Then by Corollary \ref{existencetheorem}, if $Z\subset \widehat{Y_D}$ is a (formal) subscheme of $\widehat{Y_D}$, we may choose a presentation $$\widehat{\mcl{O}_Y(-m)}^n\otimes \hat{g}^*\widehat{\mcl{O}_X(-aD)}\to \mcl{O}_{\widehat{Y_D}}\to \mcl{O}_Z$$.  Algebraizing the first map above yields (in a neighborhood of $Y_D$) a presentation for a subscheme extending $Z$.

For general $g$, we may use Chow's lemma \cite[II.5.6.1]{EGA} to obtain a quasi-projective $X$-scheme $Y'$ and a proper map $r: Y'\to Y$.  Let $Z\subset \widehat{Y_D}$ be the subscheme we wish to algebraize and $Z'=r^{-1}(Z)$ its preimage in $\widehat{Y'_D}$.  By the previous paragraph, we may algebraize $Z'$ to a closed subscheme of a Zariski-open neighborhood $U$ of $Y'_D$.  Let $Z''$ be its closure in $Y'$.  Then by the flatness of $\mcl{O}_{\widehat{Y}_D}$ over $\mcl{O}_Y$, the schemetheoretic image $r(Z'')$ is an algebraization of $Z$.  
\end{proof}
\begin{cor}\label{sectionalgebraization}
Let $k$ be a field, $g: Y\to X$  a morphism of finite type and $X$ a projective normal $k$-variety, with $D\subset X$ an ample Cartier divisor.  Let $\widehat g: \widehat{Y_D}\to \widehat D$ be the completion of $g$ at $D$.  Suppose that the dualizing sheaf $\omega_{X/k}$ has coherent cohomology supported in degrees $[-n, -m]$, with $m\geq 2$.  Then any section to $\widehat g$ extends to a Zariski-open neighborhood of $D$
\end{cor}
\begin{proof}
We may apply Corollary \ref{subschemealgebraization} to the graph of a section.
\end{proof}
\begin{cor}\label{sectionsareequal}
Let $k$ be a field, $g: Y\to X$ be a separated morphism of finite type and $X$ a projective $k$-variety, with $D\subset X$ a non-empty subscheme.  Let $\widehat g: \widehat{Y_D}\to \widehat D$ be the completion of $g$ at $D$. Then two sections to $g$ agreeing on $\widehat D$ are equal.
\end{cor}
\begin{proof}
The functor $\on{Hom}_X(X, Y)(-)$ sending an $X$-scheme $T$ to $$\on{Hom}_T(T, Y_T)=\on{Sections}(g_T)$$ is representable by a scheme locally of finite type.  Thus, letting $D_n$ be the $n$-th infinitesimal neighborhood of $D$, we have $$\on{Sections}(g_{\widehat{D}})=\varprojlim \on{Sections}(g_{D_n})=\on{Sections}(g_\mcl{D}),$$where $\mcl{D}=\on{Spec}_X \mcl{O}_{\widehat{D}}$.  But $\mcl{D}\to X$ has dense image, so any two sections agreeing on pullback to $\mcl{D}$ must agree (using that $X$ is integral).
\end{proof}

\section{Extending rational maps}\label{extension}
In this section we study the problem of extending a rational map from an open subvariety $U$ of a variety $X$ to another variety $Y$, in various settings.  Notably, we prove results when $\dim(Y)\leq \dim(X)-2$ and when $Y$ is proper and contains no rational curves.  We also give relative results (e.g. given a map $Y\to X$ and a rational section, we study the problem of extending it to a regular section) and give analogous results over general bases.
\subsection{Maps to a target of small dimension}
We first consider the case of extending a rational map $f: X\dashrightarrow Y$ where $Y$ has dimension small with respect to the base locus of $f$.  The main idea of the arguments is that the base locus of a rational map cannot be too small if $X$ has mild singularities.
\begin{prop}\label{projectiveextension}
Let $X$ be a normal $k$-variety, with $k$ a field, and let $f: Y\to X$ be a morphism of $k$-varieties with fiber dimension at most $\dim(X)-2-r$, and with $Y$ projective.  Let $U$ be a Zariski-open subset of $X$ whose complement $X\setminus U$ is $r$-dimensional, with codimension is at least $2$.  Then any section $s: U\to Y$ of $f$ extends uniquely to $X$.
\end{prop}
\begin{proof}
Let $X'$ be the (normalized) closure of the image of $s$ in $Y$; we wish to show that $f|_{X'}: X'\to X$ is an isomorphism. Let $\mcl{L}$ be a very ample line bundle on $Y$.  Then $\mcl{L}|_{X'}$ is very ample.  Let $\mcl{N}$ be the line bundle on $X$ obtained by restricting $\mcl{L}$ to $U$, and extending uniquely to $X$ (which we may do by the normality of $X$ and the fact that $X\setminus U$ has codimension at least $2$).  By the assumption on the fiber dimension of the map $f: Y\to X$, the complement $X'\setminus U$ has codimension at least $2$, so again by normality, $$\Gamma(X', \mcl{L}|_{X'})=\Gamma(U, \mcl{L}|_U)=\Gamma(X, \mcl{N}).$$
But $\mcl{L}$ was very ample, so this implies that the map $f|_{X}': X'\to X$ is indeed an isomorphism. 
\end{proof}
\begin{prop}\label{properextension}
Let $X$ be a normal, locally $\mbb{Q}$-factorial $k$-variety of dimension at least $2$, with $k$ a field, and $f: Y\to X$ a proper morphism with fiber dimension at most $\dim(X)-2-r$.  Then any rational section to $f$, defined on a Zariski-open subset $U$ of $X$ with $\dim(X\setminus U)\leq r$, extends to a regular section.
\end{prop}
\begin{proof}
Let $s$ be a rational section to $f$, as in the statement of the proposition, and let $\Gamma$ be the closure of its image.  By \cite[1.40]{debarre}, the exceptional locus of the map $\Gamma\to X$ is of pure codimension $1$ in $\Gamma$.  But the dimension of the exceptional locus is at most $\dim(X)-2$, so it must be empty; hence $\Gamma\to X$ is an isomorphism.
\end{proof}
\begin{cor}\label{nonproperextension}
Let $X$ be a normal projective $k$-variety, and $Y$ a $k$-variety such that 
\begin{enumerate}
\item $X$ is locally $\mbb{Q}$-factorial, or
\item $Y$ is quasi-projective.
\end{enumerate}
Let $D\subset X$ be an ample divisor and $U\subset X$ a Zariski-open subset containing $D$.  Then if $\dim(Y)\leq \dim(X)-2$, any map $U\to Y$ extends uniquely to a map $X\to Y$.
\end{cor}
\begin{proof}
Let $\bar Y$ be a proper compactification of $Y$, which exists by quasiprojectivity in case (2) and by Nagata \cite{conrad-nagata} in case (1).  Then by Propositions \ref{projectiveextension} and \ref{properextension}, respectively, applied to the map $\bar Y\times X\to X$, the map $U\to Y$ extends uniquely to a map $X\to \bar Y$ (here we use that $\dim(X\setminus U)=0$, by the ampleness of $D$).  We wish to show that the image of the map $\bar f: X\to \bar Y$ has empty intersection with $\bar Y\setminus Y$.  Indeed, consider $y\in \bar Y\setminus Y$.  As the $\dim(X)>\dim(Y)$, $\dim({\bar f}^{-1}(y))\geq 1$.  But then it has non-empty intersection with $D$, by the ampleness of $D$.  But this contradicts the fact that $D$ maps into $Y$.  
\end{proof}
\begin{rem}
Observe that in the above result, $Y$ need not be proper.  In fact, the proof shows that the image of the map $U\to Y$ and its extension $X\to Y$ are the same.
\end{rem}
\begin{cor}\label{smallimageextension}
Let $X$ be a normal projective $k$-variety, and $Y$ a quasi-projective $k$-variety.   Let $D\subset X$ be an ample divisor and $U\subset X$ a Zariski-open subset containing $D$.  Then any map $f: U\to Y$ with $\dim(f(D))\leq \dim(D)-2$ extends uniquely to a map $X\to Y$.
\end{cor}
\begin{proof}
Let $Y'$ be a projective compactification of the scheme-theoretic image of $f$, and resolve the rational map $U\to Y'$ to a regular map $f': X'\to Y'$, with $r: X'\to X$ proper.  By Corollary \ref{nonproperextension}, it suffices to show that $\dim(Y')\leq \dim(X)-2$; hence it suffices to show that $f(D)$ has codimension one in $Y'$.  Assume the contrary; as $Y'$ is projective, there exists a curve $C$ in $Y'$ disjoint from $f(D)$.  Now $r({f'}^{-1}(C))$ is disjoint from $D$, contradicting the ampleness of $D$. 
\end{proof}
\subsection{Maps to a target containing no rational curves}
We now consider the case of maps to a target with no rational curves.  The main idea of this section is that one may resolving a rational map via blowups, which if the source of the map is smooth, introduces many rational curves.  If the target contains no rational curves, the resolved map must contract the exceptional locus; hence the original rational map extends to a regular map.
\begin{prop}\label{noratlcurvesextension}
Let $X$ be a smooth variety over a field $k$, and let $f:Y \to X$ be a proper morphism.  If the geometric fibers of $f$ contains no rational curves, any rational section of $f$, denoted $s: X\dashrightarrow Y$, extends uniquely to a regular section $X\to Y$.
\end{prop}
\begin{proof}
Uniquess is clear, so we prove existence.  Without loss of generality, $k$ is algebraically closed.

By taking the (normalized) closure of $s(X)$ in $Y$, denoted $X'$, we find ourselives in the following situation.   We have $f: Y\to X$ a proper map with $X$ smooth and $f$ proper, and $b: X'\to X$ birational, and a section $\widetilde s$ to the map $\widetilde f: Y\times_X X'\to X'$.  That is, we have a diagram
$$\xymatrix{
Y\times_{X} X'\ar[r]^-{b_Y}\ar@/^/[d]^{\widetilde f}  & Y \ar[d]^{f}\\
X' \ar[r]^b \ar@/^/[u]^{\widetilde s}& X.
}$$
Then we wish to show that $\widetilde s$ descends a regular section to $f'$.  

Let $ X'\to Y$ be the map given by $b_Y\circ \widetilde s$.  There is a rational curve passing through the general point of an exceptional component of $b$, by e.g. \cite[Proposition 1.43]{debarre}.  Thus $b_Y\circ \widetilde s$ map contracts the fibers of $b$, as $Y$ contains no rational curves, and hence $X'$ is quasifinite over $X$.  But $f$ is proper, so $X'$ is in fact finite over $X'$; it is an isomorphism over the locus where $s$ is defined.  Hence $X'$ is isomorhpic to $X$ by Zariski's main theorem, providing us with a section as desired.
\end{proof}
\begin{lem}\label{artinflatness}
Let $A$ be an Artin local ring with residue field $k$, and $f: Y\to X$ a flat morphism of flat, finite-type $A$-schemes.  Let $X'\subset Y$ be a closed subscheme such that
\begin{enumerate}
\item the map $f|_{X'_k}: X'_k\to X_k$ is flat, 
\item the map $f|_{X'}: X'\to X$ is finite, and is flat over an open subscheme $U\subset X$, and
\item $X$ is irreducible.
\end{enumerate}
Then the map $X'\to X$ is flat.
\end{lem}
\begin{proof}
We may assume that $X, Y$ are affine, with $X=\on{Spec}(C), Y=\on{Spec}(B)$, and let $I$ be the ideal sheaf of $X'$.  The map $f$ is induced by a map of algebras $r: C\to B$.  It is enough to show that $\mcl{O}_{X'}=B/I$ is flat over $A$, by \cite[Tag 051E, Lemma 10.98.8]{stacks-project}.  Indeed, it suffices to show that $B/I$ is $A$-free.  Let $\{v_j\}_{j\in J}$ be an arbitrary lift of a $k$-basis of $\mcl{O}_{X'_k}$.  Then the induced map $$A^{\oplus J}\to B/I$$ is clearly surjective; we must show it is injective.  It is enough to show that $\on{Tor}_1^A(B/I, k)$ vanishes.  As $B$ is $A$-flat, $$\on{Tor}_1^A(\mcl{O}_\Gamma, k)=\ker(I\otimes_A k\to B\otimes_A k).$$ Whatever the kernel is, it must be annihilated by some non-zero divisor $x\in C$, as $B[1/r(x)]/I_{1/r(x)}$ is flat over $C[1/x]$ for some non-zero divisor $x$, by the generic flatness of $f|_{X'}$ and the irreducibility of $X$.  But $x$ was a non-zero divisor, so $r(x)$ is a non-zero divisor in $B$ by the flatness of $f$.  Thus $\on{Tor}_1^A(\mcl{O}_\Gamma, k)=0$ as desired.
\end{proof}
\begin{cor}\label{noratlcurvesoverS}
Let $S$ be a scheme; let $X, Y$ be finite-type flat $S$-schemes and $f: Y\to X$ a proper flat $S$-morphism.  Suppose that $X$ is $S$-smooth with geometrically connected fibers, $Y$ is $S$-flat, and that the geometric fibers of $f$ contain no rational curves.  Then any rational section to $f$, defined on an open set $U$ of $X$ which intersects each fiber of the structure morphism $X\to S$ in a non-empty open set, extends uniquely to a regular section.
\end{cor}
\begin{proof}
We immediately reduce to the case of $S$ affine; by a standard limit argument, we may reduce to the case where $S$ is of finite type over $\mbb{Z}$.  Let $s$ be a rational section, and let $\Gamma$ be the (scheme-theoretic) closure of its image.  By Proposition \ref{noratlcurvesextension}, the map $\Gamma\to X$ is bijective. Thus over each $\bar k$-point $v$ of $S$, $\Gamma_v\to X_v$ is an isomorphism, by Zariski's main theorem.  To check that $\Gamma\to X$ is an isomorphism, it suffices to consider the case where $S$ is the spectrum of a local Artin ring $A$, with algebraically closed residue field $k'$.  Again by Proposition \ref{noratlcurvesextension}, the map $\Gamma_k\to X_k$ is an isomorphism. 

Thus it suffices to show that $\Gamma\to X$ is flat.  But this is immediate from Lemma \ref{artinflatness}.
\end{proof}
Similar arguments, replacing the use of Proposition \ref{noratlcurvesextension} with Propositions \ref{properextension} or \ref{projectiveextension} give the following results over a general base:
\begin{cor}
Let $S$ be a scheme and let $X, Y$ be finite-type flat $S$-schemes.  Let $f: Y\to X$ be a proper flat $S$-morphism, and suppose
\begin{enumerate}
\item The geometric fibers of the structure map $X\to S$ are normal, and the geometric fibers of the structure map $Y\to S$ are projective, or
\item The geometric fibers of the structure map $X\to S$ are normal and locally $\mbb{Q}$-factorial.
\end{enumerate}
Suppose further that the fibers of $f$ have dimension at most $\dim(X)-2-r$, and that $U\subset X$ is an open subscheme with $\dim(X_s\setminus(U\cap X_s))\leq r$ for all points $s\in S$.  Then any section to $f$ defined on $U$ extends uniquely to all of $X$.
\end{cor}
Using Corollary \ref{nonproperextension} instead, we obtain:
\begin{cor}
Let $S$ be a scheme and let $X, Y$ be finite-type flat $S$-schemes.  Suppose
\begin{enumerate}
\item The geometric fibers of the structure map $X\to S$ are normal, and the geometric fibers of the structure map $Y\to S$ are quasi-projective, or
\item The geometric fibers of the structure map $X\to S$ are normal and locally $\mbb{Q}$-factorial.
\end{enumerate}
Suppose further that the geometric fibers of the structure map $Y\to S$ have dimension at most $\dim(X)-2-r$, and that $U\subset X$ is an open subscheme with $\dim(X_s\setminus(U\cap X_s))\leq r$ for all points $s\in S$.  Then any map $U\to Y$ over $S$ extends uniquely to a map $X\to Y$ (again over $S$).
\end{cor}
\begin{cor}
Let $S$ be a scheme; let $X$ be a smooth finite-type $S$-schemes and $Y$ a proper flat $S$-scheme whose geometric fibers over $S$ containing no rational curves.  Then any rational map $X\dashrightarrow Y$ extends uniquely to a regular map.
\end{cor}
\begin{proof}
This is immediate from Corollary \ref{noratlcurvesoverS}, by considering the projection map $X\times_S Y\to X$.
\end{proof}

\section{Vanishing results and deformation theory}\label{deformation theory}
The purpose of this chapter is to study the following question:  given a varieties $X$ and $Y$, and an ample Cartier divisor $D\subset X$, when does a map $D\to Y$ extend (uniquely) to a map $\widehat D\to Y$ where $\widehat D$ is the formal scheme obtained by completing $X$ at $D$?  This is a purely deformation-theoretic question, to which we give separate answers in characteristic zero and in positive characteristic; even in characteristic zero, however, the result relies on a positive characteristic argument.  The characteristic zero results, which are of a rather different flavor, appear in Section \ref{charzeroextension}.  
\subsection{Deformation-theoretic and cohomological preliminaries}
The main results we will need in this section are the following:
\begin{thm}
Let $f: Y\to X$ be a smooth morphism of schemes, and let $D\subset X$ be a closed lci subscheme with ideal sheaf $\mcl{I}_D$.  Consider the Cartesian diagram
$$\xymatrix{
Y_D \ar[d]^{f_D} \ar[r] & Y\ar[d]^f\\
D \ar[r] & X.
}$$
Let $s: D\to Y$ be a section to the map $f_D: Y_D\to D.$  Let $D_2\subset X$ be the subscheme defined by the ideal sheaf $\mcl{I}_D^2$.  Then there is a natural class $o(s)\in \on{Ext}^1(\mcl{N}_{D/X}, s^*T_{Y/X})$ whose vanishing is equivalent to the existence of an extension of $s$ to $D_2$; such extensions are a torsor for $\on{Hom}(\mcl{N}_{D/X}, s^*T_{Y/X}).$
\end{thm}
\begin{cor}
Let $X, D, D_2$ be schemes over a field $k$, and let $Y$ be an arbitrary smooth $k$-scheme.  Then if $s: D\to Y$ is a morphism, there is a natural obstruction class $o(s)\in \on{Ext}^1(\mcl{N}_{D/X}, s^*T_{Y/k})$ whose vanishing is equivalent to the existence of an extension of $s$ to $D_2$; such extensions are a torsor for $\on{Hom}(\mcl{N}_{D/X}, s^*T_{Y/k}).$
\end{cor}
Here $\mcl{N}_{D/X}$ is the normal bundle of $D$ in $X$, and $T_{Y/X}$ is the relative tangent bundle of $Y$ over $X$.  Loosely speaking, the idea is that a deformation of the map exhibits a rule for sending normal vectors to  $D$ to tangent vectors in $Y$.

More generally, we have:
\begin{thm}\label{maindefthm}
Let $f: Y\to X$ be a smooth morphism of schemes, and let $D\subset X$ be a closed lci subscheme with ideal sheaf $\mcl{I}_D$.  Consider the Cartesian diagram
$$\xymatrix{
Y_D \ar[d]^{f_D} \ar[r] & Y\ar[d]^f\\
D \ar[r] & X.
}$$
Let $s: D\to Y$ be a section to the map $f_D: Y_D\to D.$  Let $D_n\subset X$ be the subscheme defined by the ideal sheaf $\mcl{I}_D^n$; let $s_n: D_n\to Y$ be a section to $f_{D_n}: Y_{D_n}\to D_n$ extending $s$.  Then there is a natural class $o(s)\in \on{Ext}^1(s^*\Omega^1_{Y/X}, \mcl{I}_D^n/\mcl{I}_D^{n+1})$ whose vanishing is equivalent to the existence of an extension of $s$ to $D_{n+1}$; such extensions are a torsor for $\on{Hom}(s^*\Omega^1_{Y/X}, \mcl{I}_D^n/\mcl{I}_D^{n+1}).$
\end{thm}
\begin{proof}
This is well-known, but we include a sketch proof for the sake of completeness.  Suppose we wish to extend $s_n$ to a map $s_{n+1}: D_{n+1}\to Y$.  Such an extension is the same as filling in the dotted arrow in the diagram
$$\xymatrix{
& & & s^{-1}\mcl{O}_Y \ar[d]^{s_n} \ar@{.>}[ld]&\\
0 \ar[r] &\mcl{I}_D^n/\mcl{I}_D^{n+1} \ar[r]  & \mcl{O}_{D_{n+1}} \ar[r]& \mcl{O}_{D_n} \ar[r] & 0.
}$$
Any two such lifts differ by an element of $\on{Der}_{\mcl{O}_X}(s^{-1}\mcl{O}_Y, \mcl{I}_D^n/\mcl{I}_D^{n+1})$; i.e. if a lift exists, the set of lifts are a torsor for $\on{Hom}(s^*\Omega^1_{Y/X}, \mcl{I}_D^n/\mcl{I}_D^{n+1})$ as desired.  Such lifts do exist locally on $X$, by the smoothness of $Y$.  So we may choose local lifts $s_{n+1}^i$ over some cover $\{U_i\}$ of $X$.  Let $U_{ij}=U_i\cap U_j$.  Now  $\{s_{n+1}^i|_{U_{ij}}-s_{n+1}^j|_{U_{ij}}\}$, viewed as a set of maps $$s^*\Omega^1_{Y/X}|_{U_{ij}}\to \mcl{I}_D^n/\mcl{I}_D^{n+1}|_{U_{ij}},$$  is a \v{C}ech cocycle representative for an element of $$\on{Ext}^1(s^*\Omega^1_{Y/X}, \mcl{I}_D^n/\mcl{I}_D^{n+1})=H^1(X, \ul{\Hom}(s^*\Omega^1_{Y/X}, \mcl{I}_D^n/\mcl{I}_D^{n+1}));$$ this is $o(s)$.  Indeed, if $o(s)=0$, we may refine the cover $\{U_i\}$ and find a cocycle exhibiting $\{s_{n+1}^i|_{U_{ij}}-s_{n+1}^j|_{U_{ij}}\}$ as a coboundary; modifying the $s_{n+1}^i$ by this cocycle, we find that a lift exists.  Likewise, one may easily construct such a cocycle from a lift.
\end{proof}

In particular, deformations of a map $D\to Y$ will exist if $$\on{Ext}^1(s^*\Omega^1_{Y/X}, \mcl{I}_D^n/\mcl{I}_D^{n+1})=0,$$ and they will be unique if $$\on{Hom}(s^*\Omega^1_{Y/X}, \mcl{I}_D^n/\mcl{I}_D^{n+1})=0.$$  Thus we search for conditions on $D, Y$ under which these two groups vanish.  Such conditions will come from positivity of both $D$ and $\Omega^1_{Y/X}$; for example, in characteristic zero it will suffice that $D$ be ample and $\Omega^1_{Y/X}$ be nef.  

In the case that $D$ is smooth and $\dim(Y)<\dim(D)$, we sketch an easy characteristic zero argument for the required vanishing.

Recall the Le Potier vanishing theorem \cite[Theorem 7.3.5]{lazarsfeld2004positivity2}:
\begin{thm}[Le Potier]\label{lepotier}
Let $E$ be an ample vector on a smooth projective variety $X$ over a field of characteristic zero, with $\dim(X)=n$.  Then $$H^i(X, \omega_X\otimes \bigwedge^aE)=0$$ for $a>0, i>e-a$, and $$H^i(X, \Omega^p_X\otimes E)=0$$ for $i+p\geq n-e.$
\end{thm}

Observe that if $Y$ is smooth, $$\on{Ext}^i(s^*\Omega^1_{Y/X}, \mcl{I}_D^n/\mcl{I}_D^{n+1})=H^i(D, \mcl{O}_D(-nD)\otimes s^*T_{Y/X}).$$
If $D$ is ample and $\Omega^1_{Y/X}$ is a nef vector bundle, the vector bundle $\mcl{O}_D(-nD)\otimes s^*T_{Y/X}$ is anti-ample.  Thus if $D$ is smooth and $\dim(Y)<\dim(D)$, the Le Potier vanishing theorem combined with Serre duality immediately implies the required vanishing.  So we have shown
\begin{thm}\label{lamecharzeroresult}
Let $X$ be a projective variety over a field $k$ of characteristic zero, and let $D\subset X$ be a smooth ample Cartier divisor, with $\dim(X)\geq 3$.  Let $Y$ be a smooth variety over $k$ with $\Omega^1_Y$ nef and with $\dim(Y)<\dim(D)$, and let $f: D\to Y$ be a morphism.  Then $f$ extends uniquely to a morphism $\widehat D\to Y$, where $\widehat D$ is the formal scheme obtained by completing $X$ at $D$.
\end{thm}
So the main work of this section will be to deal with the case where $D$ is not smooth.  The obvious obstruction to imitating the proof of Theorem \ref{lamecharzeroresult} for singular $D$ is that the le Potier vanishing theorem does not hold for arbitrary varieties.  We will avoid this by proving a weaker, positive characteristic result for arbitrary $D$ (that is, Theorem \ref{frobeniusvanishing}), and then leveraging the smoothness of $X$ to deduce an improvement of Theorem \ref{lamecharzeroresult} where $D$ may have arbitrary singularities.

Indeed, there are certain results we can only obtain in characteristic zero for smooth $D$, which we state here.
\begin{thm}
Let $X$ be a projective variety, and let $D\subset X$ be a smooth lci subscheme, with ample normal bundle, and with $\dim(D)\geq 2$.  Let $\widehat D$ the formal scheme obtained by completing $X$ at $D$.  Let $Y$ be a smooth variety with Nakano semi-positive cotangent bundle.  Then given a morphism $f: D\to Y$,
\begin{itemize}
\item there is at most one extension of $f$ to a morphism $\widehat D\to Y$ if $\dim(D)\geq 1$, and
\item such an extension exists as long as $\dim(D)\geq 2$.
\end{itemize}
\end{thm} 
\begin{proof}
The main observation here is that the sheaf $f^*\Omega^1_Y\otimes \mcl{N}_{D/X}$ is Nakano-positive, hence satisfies the required vanishing by Nakano's vanishing theorem \cite[after example 7.3.17]{lazarsfeld2004positivity2}.
\end{proof}
\subsection{Frobenius amplitude}
We first consider the case of positive characteristic.  Let $k$ be a field with positive characteristic, and $X$ a $k$-variety.  If $f:X\to S$ is a morphism over a positive characteristic base, recall the definition of the relative Frobenius.  If $Y$ is any $k$-scheme, we let $\on{Frob}_p$ be the absolute Frobenius morphism of $Y$ (which is not a morphism of $k$-schemes).  There is a natural diagram
$$\xymatrix@R=3em@C=10em{
X \ar@{.>}[rd]|{F_{X/S}} \ar[rrd]^{\on{Frob}_p} \ar[ddr]^f& & \\
 & X^{(p)} \ar[r] \ar[d]^{f^{(p)}} & X\ar[d]^f\\
 & S \ar[r]^{\on{Frob}_p}& S
}$$
where the square on the right is Cartesian, with $X^{(p)}:=X\times_{S, \on{Frob}_p} S$ and the map $F_{X/S}: X\to X^{(p)}$ is defined via the universal property of the Carterisan product.  

The main idea of this section is that if $\mcl{E}$ is a vector bundle on $X$ with some positivity property, $E^{(p)}:=\on{Frob}_p^*E$ has increased positivity; similarly, if $E'$ is a vector bundle on $X^{(p)}$ with some positivity property, $F_{X/S}^*E'$ has increased positivity.

Following, \cite{arapura-f-amplitude}, we make the following definition to measure the asymptotic positivity of $E^{(p^k)}$, as $k\to \infty$:
\begin{defn}[f-amplitude]
Let $E$ be a vector bundle on a $k$-scheme $X$.  Then if $\on{char}(k)=p>0$, the f-amplitude of $E$, denoted $\phi(E)$ is the least integer $i_0$ such that $$H^i(X, \mcl{F}\otimes E^{(p^k)})=0 \text{ for } k\gg 0$$ for all coherent sheaves $\mcl{F}$ on $X$ and $i> i_0.$  If $\on{char}(k)=0$, $\phi(E)$ is defined to be the the infimum of $$\on{max}_{\mfk{q}\in A} \phi(E_\mfk{q}),$$ where $A$ is a finite-type $\mbb{Z}$-scheme, $(\mcl{X}, \mcl{E})$ is a model of $(X, E)$ over $A$, and $\mfk{q}$ ranges over all closed points of $A$.
\end{defn}
The main result we will need is the following bound on the f-amplitude of an ample vector bundle \cite[Theorem 6.1]{arapura-f-amplitude}:
\begin{thm}\label{arapurabound}
Let $k$ be a field of characteristic zero, $X$ be a projective $k$-scheme, and $\mcl{E}$ an ample vector bundle on $X$.  Then $$\phi(\mcl{E})<\on{rk}(\mcl{E}).$$
\end{thm}
\begin{rem}
It is presently unknown whether the bound of Theorem \ref{arapurabound} holds in positive characteristic; we expect that it does.  The key ingredient in the proof is a natural resolution of the functor $\mcl{E}\mapsto \mcl{E}^{(p)}$ by Schur functors (see \cite[p. 235]{carter-lusztig}); an analogous resolution of the functor $\mcl{E}\mapsto \mcl{E}^{(p^n)}$ would suffice to give the result.
\end{rem}
\begin{rem}
One may deduce versions of Le Potier's vanishing theorem (Theorem \ref{lepotier}) and many other interesting vanishing theorems from Theorem \ref{arapurabound} and the methods of Deligne-Illusie \cite{deligne-illusie}.  This is the purpose of Donu Arapura's beautiful papers \cite{arapura-f-amplitude, arapura-partial-regularity, arapura-ultraproducts}.
\end{rem}
We will also make heavy usage of the following analogue of Le Potier's vanishing theorem \cite[Theorem 8.2]{arapura-f-amplitude}, proven using the methods of \cite{deligne-illusie}:
\begin{thm}\label{arapuravanishing}
Let $k$ be a perfect field of characteristic $p>n$, and let $X$ be a smooth $n$-dimensional projective variety over $k$.  Let $\mcl{E}$ be a vector bundle on $X$, and suppose that $X$ lifts to $W_2(k)$.  Then $$H^i(X, \Omega^j_X\otimes \mcl{E})=0$$ for $i+j>n+\phi(\mcl{E}).$
\end{thm}
\subsection{Extending after composition with Frobenius}\label{extendingafterfrobeniussection}
The main result of this section is the following extension theorem.  The use of these sorts of positive-characteristic vanishing results was suggested to the author by Bhargav Bhatt.
\begin{thm}\label{frobeniusvanishing}
Let $X$ be variety over a field $L$ of characteristic $p$, and let $D\subset X$ be a Cartier divisor whose dualizing complex $K_D$ is supported in degrees $[-\dim(D), -r]$, and whose normal bundle is ample.  Suppose $\dim(X)\geq 3$.  Let $\widehat D$ be the formal scheme obtained by completing $X$ at $D$.  Let $f: D\to Y$ be a morphism, with $Y$ a smooth $k$-variety.  Suppose that $\phi(f^*\Omega^1_Y\otimes \mcl{N}_{D/X})<r-1$.  Let $\widetilde f^{(p^k)}=F_{Y/L}^k\circ f: D\to Y^{(p^k)}$.  Then for $k\gg0, \widetilde f^{(p^k)}$ extends uniquely to a morphism $\widehat D\to Y^{(p^k)}$.
\end{thm}
We prove will prove this theorem soon.  Essentially identical arguments give a version of this theorem for sections:
\begin{thm}
Let $X$ be variety over a field $L$ of characteristic $p$, and let $D\subset X$ be a Cartier divisor whose dualizing complex $K_D$ is supported in degrees $[-\dim(D), -r]$, and whose normal bundle is ample.  Suppose  $\dim(X)\geq 3$.  Let $\widehat D$ be the formal scheme obtained by completing $X$ at $D$.  Let $g: Y\to X$ be a smooth morphism, and $f: D\to Y_D$ a section to $g_D: Y_D\to D$.  Suppose that $\phi(f^*\Omega^1_{Y/X}\otimes \mcl{N}_{D/X})<r-1$.  Let $g^{(p^k)}$ be the map defined via the fiber square
$$\xymatrix{
Y^{(p^k)} \ar[r] \ar[d]^-{g^{(p^k)}} & Y \ar[d]^g\\
X \ar[r]^{\on{Frob}_p^k} & X,}$$  
and let $\widetilde f^{(p^k)}$ be the section to $g^{(p^k)}_D$ induced from $f$ by the universal property of the fiber product $Y^{(p^k)}$. Then for $k\gg0, \widetilde f^{(p^k)}$ extends to a morphism $\widehat D\to Y^{(p^k)}$.
\end{thm}
The proof of hthis theorem is only notationally more complicated than that of Theorem \ref{frobeniusvanishing}, so we omit it.

Before proving Theorem \ref{frobeniusvanishing}, we need an elementary lemma.
\begin{lem}\label{frobeniusfactors}
Let $X$ be a $k$-scheme, with $k$ of characteristic $p>0$, and let $D\subset X$ be an effective Cartier divisor, defined by an ideal sheaf $\mcl{O}(-D)$.  Let $D_n$ be the Cartier divisor defined by $\mcl{O}(-nD)$.  Then the relative Frobenius map $F_{D/k}: D\to D^{(p)}$ factors through the inclusion $\iota_p: D\hookrightarrow D_p$, i.e. there exists a natural map $\tau_p: D_p\to D^{(p)}$ so that $F_{D/k}=\tau_p\circ \iota_p$.
\end{lem}
\begin{proof}
We may assume that $X=\on{Spec}(A)$ is affine, and $D$ is principle, cut out by some element $f\in A$.  Then $\on{Frob}_p: D\to D$ is given by the $p$-th power map on $A/f$, which factors through the natural map $A/f^p\to A/f$.  That is, the diagram defining the relative Frobenius map factors as
$$\xymatrix@R=3em@C=10em{
D \ar@{.>}[rd]|{F_{X/S}} \ar[rrd]^{\on{Frob}_p} \ar[ddr]^f \ar@{^(->}[r]^{\iota_p}& D_p \ar[rd]& \\
 & D^{(p)} \ar[r] \ar[d]^{f^{(p)}} & D\ar[d]^f\\
 & \on{Spec}(k) \ar[r]^{\on{Frob}_p}& \on{Spec}(k).
}$$
The induced map $D_p\to D$ and the structure map $D_p\to \on{Spec}(k)$ give the desired map $\tau_p: D_p\to D^{(p)}$, via the universal property of the Cartesian product.
\end{proof}
\begin{proof}[Proof of Theorem \ref{frobeniusvanishing}]
Without loss of generality, $L$ is perfect.  Let $f: D\to Y$ be a morphism, as in the statement of the theorem.  Then by the ``universal commutativity of Frobenius," i.e. the commutativity of the diagram
$$\xymatrix@C=3em{
D \ar[r]^{F_{D/k}} \ar[d]^f \ar@/^2pc/[rr]^{\on{Frob}_p}& D^{(p)} \ar[d]^{f^{(p)}} \ar[r]& D\ar[d]^f\\
Y \ar[r]^{F_{Y/k}} \ar@/_2pc/[rr]_{\on{Frob}_p}& Y^{(p)} \ar[r] &  Y
}$$
we have that $\widetilde f^{(p^k)}:=\on{Frob}_p^k\circ f$ is also equal to $f\circ \on{Frob}_p^k$.  By Lemma \ref{frobeniusfactors}, the morphism $\widetilde f^{(p^k)}$ admits a natural extension to $D_{p^k}$.  We claim that for $k\gg 0$, this morphism extends naturally to a morphism $D_{p^k+s}\to Y^{(p^k)}$ for all $s$.  Indeed, it will suffice to take $k$ large enough so that $$H^{\dim(D)-\epsilon}(D, K_D(sD)\otimes \on{Frob}_p^{k*}(f^*\Omega^1_Y\otimes \mcl{N}_{D/X}))=0$$ for $\epsilon=0, 1$ and all $s\geq 0$.  Such a $k$ exists by the assumption on the $f$-amplitude of $f^*\Omega^1_Y\otimes \mcl{N}_{D/X}$, and the amplitude of $\mcl{N}_{D/X}=\mcl{O}_D(D)$.  

Indeed, extending the morphism $D_{p^k+s}\to Y^{(p^k)}$ to a morphism $D_{p^k+s+1}\to Y^{(p^k)}$ is the same as extending the composite morphism $s: D_{p^k+s}\to Y^{(p^k)}\to Y$ (which is not a morphism of $k$-schemes, but rather a morphism ``over" $\on{Frob}_p: k\to k$), by the definition of $Y^{(p^k)}$.  The obstruction to such an extension lies in $$\on{Ext}^1(\widetilde f^{(p^k)*}\Omega^1_{Y}, \mcl{I}_D^{p^k+s}/\mcl{I}_D^{p^k+s+1})=H^1(D,  \widetilde f^{(p^k)*}T_{Y}\otimes \mcl{O}_D((-p^k-s)D)),$$ and assuming the obstruction vanishes, such extensions are a torsor for $$\on{Hom}(\widetilde f^{(p^k)*}\Omega^1_{Y}, \mcl{I}_D^{p^k+s}/\mcl{I}_D^{p^k+s+1})=H^0(D,  \widetilde f^{(p^k)*}T_{Y}\otimes \mcl{O}_D((-p^k-s)D)),$$
by Theorem \ref{maindefthm}; we wish to show that both of these groups vanish for $k\gg0$.  But $$\widetilde f^{(p^k)*}T_{Y}\otimes \mcl{O}_D((-p^k-s)D)=\on{Frob}_p^{k*}(f^*T_Y\otimes\mcl{O}_D(-D))\otimes \mcl{O}_D(-sD).$$  Recall also that $\mcl{O}_D(-D)=\mcl{N}_{D/X}^\vee$.  But by Grothendieck duality, \begin{equation*}\begin{split} H^i(D, \on{Frob}_p^{k*}(f^*T_Y\otimes\mcl{O}_D(-D))\otimes \mcl{O}_D(-sD))=\\ \qquad H^{\dim(D)-i}(D, K_D(sD)\otimes \on{Frob}_p^{k*}(f^*\Omega^1_Y\otimes \mcl{N}_{D/X})),\end{split}\end{equation*} 
which is zero by our assumption on $k$.
\end{proof}
\begin{rem}\label{zeroonDpk}
Observe that, from the proof of Theorem \ref{frobeniusvanishing}, the induced map $g_{p^k}: D_{p^k}\to Y^{(p^k)}\to Y$ factors through the natural ``$p^k$-th power map" $D_{p^k}\to D$. Thus in particular the map $g_{p^k}^*\Omega^1_Y\to \Omega^1_{D_{p^k}}$ is zero.
\end{rem}
Recall that given a morphism $f: D\to Y$, the goal of this section was to extend $\widetilde f^{p^k}:=\on{Frob}_p^k\circ f: D\to Y^{(p^k)}$ to  a morphism $\widehat D\to Y^{(p^k)}$.  Theorem \ref{frobeniusvanishing} reduces this problem to estimating the f-amplitude of the sheaf $f^*\Omega^1_Y\otimes \mcl{N}_{D/X}$.  Our main result in positive characteristic will be when $\Omega^1_Y$ is f-semipositive (which we will recall from \cite{arapura-f-amplitude}) and $D$ has some positivity properties (e.g. $D$ is an ample divisor).   To introduce the notion of f-semipositivity, we will need to recall the classical notion of Casteluovo-Mumford regularity. 
\begin{defn}[Castelnuovo-Mumford Regularity]
Let $X$ be a variety, and $\mcl{O}(1)$ an ample line bundle on $X$.  The \emph{Castelnuovo-Mumford regularity} of a coherent sheaf $\mcl{F}$ on $X$ with respect to $\mcl{O}(1)$ is the least $n$ such that $$H^i(X, \mcl{F}(n-i))=0$$ for all $i>0$.
\end{defn}
\begin{defn}[f-semipositivity, {\cite[Definition 3.8]{arapura-f-amplitude}}]\label{f-semipositivity-def}
Let $k$ be a field of characteristic $p>0$, and let $S$ be a projective $k$-scheme.  We say that a coherent sheaf $\mcl{F}$ on $S$ is f-semipositive if the Castelnuovo-Mumford regularity of the Frobenius pullbacks $\{\mcl{F}^{(p^k)}\}$ with respect to a fixed ample line bundle on $S$ are bounded.  This notion is independent of the choice of ample line bundle.

If $k$ is a field of characteristic zero, $S$ is a projective $k$-scheme, and $\mcl{F}$ is a coherent sheaf on $S$, we say that $\mcl{F}$ is f-semipositive if there exists a model of $(S, \mcl{F})$ over a finite-type $\mbb{Z}$-scheme so that the positive-characteristic fibers of $\mcl{F}$ are f-semipositive.
\end{defn}
\begin{rem}
A priori, the definition of f-semipositivity appears to rely on a choice of ample line bundle; however, the class of f-semipositive sheaves is independent of this choice \cite[Corollary 3.11]{arapura-f-amplitude}.  
\end{rem}
The main result about f-semipositivity which we will use is \cite[Theorem 4.5]{arapura-f-amplitude}:
\begin{thm}\label{tensorproductbound}
Let $\mcl{E}$ and $\mcl{F}$ be vector bundles with $\mcl{F}$ f-semipositive.  Then $$\phi(\mcl{E}\otimes \mcl{F})\leq \phi(\mcl{E}).$$
\end{thm}
\begin{cor}
Let $X$ be a smooth projective variety over a field $k$ and let $D\subset X$ be a Cartier divisor with ample normal bundle.  Let $Y$ be a smooth $k$-variety with $\Omega^1_Y$ f-semipositive.  Then for any morphism $f: D\to Y$, $$\phi(f^*\Omega^1_Y\otimes \mcl{N}_{D/X})=0.$$
\end{cor}
\begin{proof}
By \cite[Proposition 3.10]{arapura-f-amplitude}, $f^*\Omega^1_Y$ is f-semipositive.  Thus by Theorem \ref{tensorproductbound}, $\phi(f^*\Omega^1_Y\otimes \mcl{N}_{D/X})=\phi(\mcl{N}_{D/X})$.  But for a line bundle $\mcl{L}$, amplitude is the same as the condition that $\phi(\mcl{L})=0$ (this is elementary, but for a proof, see \cite[Lemma 2.4]{arapura-f-amplitude}).  Hence $\phi(\mcl{N}_{D/X})=0$, completing the proof.
\end{proof}
For examples of varieties $Y$ with f-semipositive cotangent bundle, see Section \ref{f-semipositive-section}.
\subsection{Extension results in characteristic zero and for liftable varieties}\label{charzeroextension}
\subsubsection{Uniqueness of Extensions}
Let $k$ be a field, and let $X$ be a smooth projective $k$-variety.  Let $D\subset X$ be an ample divisor, and $Y$ a smooth projective $k$-variety, and $f: D\to Y$ a map.  We now study the extent to which extensions of $f$ to a morphism $X\to Y$ are unique.  Our main theorem is:
\begin{thm}\label{unique-extensions}
Let $k$ be a field, and let $X$ be a smooth projective $k$-variety with $\dim(X)\geq 2$.  Let $D\subset X$ be an ample divisor, and let $Y$ be a smooth $k$-variety.  Let $f: X\to Y$ be a morphism such that $$\phi(f^*\Omega^1_Y(D))<\dim(D).$$ Then if 
\begin{enumerate}
\item $\on{char}(k)=0$, or
\item $k$ is perfect of characteristic $p>\dim(X)$ and $X$ lifts to $W_2(k)$, 
\end{enumerate}
there is at most one extension of $f|_D: D\to Y$ to a map $X\to Y$, namely $f$ itself.
\end{thm}
\begin{proof}
Let $\widehat D$ be the formal scheme obtained by completing $X$ at $D$.  By Corollary \ref{sectionsareequal}, it suffices to show that an extension of $f|_D$ to $\widehat D$ is unique.  By standard deformation theory (Theorem \ref{maindefthm}), it is enough to show that $$\on{Hom}(f^*\Omega^1_Y|_D, \mcl{I}_D^n/\mcl{I}_D^{n+1})=H^0(D, (f|_D^*\Omega^1_Y)^\vee\otimes \mcl{O}_X(-nD)|_D)=0$$ for all $n\geq 1$, where $\mcl{I}_D$ is the ideal sheaf of $D$.

Observe that \begin{equation}\label{phibound}\phi(f^*\Omega^1_Y(nD))\leq \phi(f^*\Omega^1_Y(D))<\dim(D),\end{equation} because tensoring with an ample line bundle does not increase f-amplitude, by Theorem \ref{tensorproductbound} (using that ample line bundles are f-semipositive \cite[Proposition A.2]{arapura-f-amplitude}).  

Now the short exact sequence $$0\to f^*(\Omega^1_Y)^\vee((-n-1)D)\to (f^*\Omega^1_Y)^\vee(-nD)\to (f|_D^*\Omega^1_Y)^\vee\otimes \mcl{O}_X(-nD)|_D)\to 0$$ induces a long exact sequence in cohomology, so to show that $$H^0(D, (f|_D^*\Omega^1_Y)^\vee\otimes \mcl{O}_X(-nD)|_D)=0$$ as desired, it is enough to show that $$H^0(X,  (f^*\Omega^1_Y)^\vee(-nD))=H^1(X,  f^*(\Omega^1_Y)^\vee((-n-1)D))=0.$$  But this is immediate from our estimate \ref{phibound} and the vanishing theorem \ref{arapuravanishing}.
\end{proof}
\subsubsection{Existence of Extensions}
We now deduce consequences of Section \ref{extendingafterfrobeniussection} in characteristic zero and for varieties over a field $L$ of characteristic $p>0$, which lift to $W_2(L)$.  Our main theorem is:
\begin{thm}\label{charzeroextensionthm}
Let $k$ be a field of characteristic zero.  Let $X$ be a smooth $k$-variety and $D\subset X$ an ample Cartier divisor, with $\dim(X)\geq 3$.  Let $Y$ be a smooth $k$-variety and let $f: D\to Y$ be a morphism such that $$\phi(\mcl{N}_{D/X}\otimes f^*\Omega^1_Y)<\dim(D)-1.$$   Suppose further that  
\begin{enumerate}
\item $\dim(Y)<\dim(D)$, or
\item $Y$ is quasi-projective and $\dim(\on{im}(f))<\dim(D)-1$, or
\item $Y$ is proper and admits a model which is finite type over $\mbb{Z}$ and whose geometric fibers contain no rational curves.
\end{enumerate}   
Then $f$ extends uniquely to a morphism $X\to Y$.
\end{thm}
In fact, one has the following interesting theorem in characteristic zero, which is an analogue of Theorem \ref{frobeniusvanishing}, without the global hypotheses on $Y$.  
\begin{thm}\label{formalextensioncharzero}
Let $k$ be a field of characteristic zero.  Let $X$ be a smooth $k$-variety and $D\subset X$ an ample divisor, with $\dim(X)\geq 3$.  Let $Y$ be a smooth $k$-variety and let $f: D\to Y$ be a morphism such that  $$\phi(\mcl{N}_{D/X}\otimes f^*\Omega^1_Y)<\dim(D)-1.$$   Then $f$ extends uniquely to a morphism $\widehat D\to Y$, where $D$ is the formal scheme obtained by completing $X$ at $D$.
\end{thm}
The proof is analogous to that of Theorem \ref{charzeroextensionthm}, but relies on certain unpublished vanishing results for $H^i(\widehat D, \widehat{\Omega^q_X}\otimes E)$, where $E$ has bounded Frobenius-amplitude, along the lines of \cite[Corollary 8.6]{arapura-f-amplitude} (here $\widehat{\Omega^q_X}$ denotes the pullback of $\Omega^q_X$ to $\widehat D$).  As we will not need this result, we omit the proof.  We observe that one may rather cheaply prove a local version of Theorem \ref{formalextensioncharzero} in the case that $D$ is smooth:
\begin{thm}\label{smoothcharzeroresult}
Let $k$ be a field of characteristic zero.  Let $D\hookrightarrow \widehat D$ be an embedding of a smooth proper $k$-variety in a smooth formal scheme, so that the ideal sheaf $\mcl{I}_D$ is an ideal of definition for $\widehat D$, with $\dim(D)\geq 2$.   Let $Y$ be a smooth $k$-variety and $f: D\to Y$ a morphism such that $$\phi((\mcl{I}_D^n/\mcl{I}_D^{n+1})^\vee\otimes f^*\Omega^1_Y)<\dim(D)-1$$ for all $n$.  Then $f$ extends uniquely to a morphism $\widehat D\to Y$.
\end{thm}
\begin{proof}
We must show that $$\on{Ext}^i(f^*\Omega^1_Y, \mcl{I}_D^n/\mcl{I}_D^{n+1})=H^i(D, (f^*\Omega^1_Y)^\vee\otimes \mcl{I}_D^n/\mcl{I}_D^{n+1})=0$$ for $i=0,1$.  But this is immediate from the estimate on the f-amplitude of the relevant sheaves, combined with the vanishing theorem \cite[Corollary 8.6]{arapura-f-amplitude}.
\end{proof}
\begin{rem}
Observe that Theorem \ref{smoothcharzeroresult} subsumes Theorem \ref{lamecharzeroresult}, using the bound in Theorem \ref{arapurabound} to see that the hypothesis on $\phi((\mcl{I}_D^n/\mcl{I}_D^{n+1})^\vee\otimes f^*\Omega^1_Y)$ is satisfied if $D$ is ample and $\Omega^1_Y$ is nef of rank less than $\dim(D)$.
\end{rem}

The idea of the proof of Theorem \ref{charzeroextensionthm} will be to spread the triple $(X, D, Y)$ over a finite type $\mbb{Z}$-scheme, and use Theorem \ref{frobeniusvanishing} to obtain extensions of the positive-characteristic fibers of the morphism $f: D\to Y$ to a morphism from $X$, after suitable composition with a power of Frobenius.  We will then show that in fact it was unnecessary to compose with Frobenius at all; this will allow us to lift the extensions to characteristic zero.

Before proceeding with the proof, we will need a lemma which tells us when a morphism factors through Frobenius.
\begin{lem}\label{factorsthroughfrobenius}
Let $k$ be a perfect field of characteristic $p>0$, and let $X$ be a normal, reduced $k$-variety.  Let $Y$ be an arbitrary $k$-variety.  Then a morphism $f: X\to Y$ factors through $\on{Frob}_p: Y\to Y$ if and only if the induced map $f^*\Omega^1_Y\to \Omega^1_X$ is zero; furthermore, this factorization is unique.
\end{lem}
\begin{proof}
Without loss of generality, $X$ is connected, and hence by normality, integral.  It suffices to consider the case of $X, Y$ affine, with $X=\on{Spec}(A)$ and $Y=\on{Spec}(B)$, $B=k[t_1, \cdots, t_n]/I$; let $K=\on{Frac}(A)$.  Then $f$ is determined by the images of the elements $t_i$ in $A$, $$f: t_i\mapsto a_i\in A.$$  The assumption on the map $f^*\Omega^1_Y\to \Omega^1_X$ implies that $d(a_i)\in \Omega^1_{A/k}$ is zero for all $i$.  In particular, $d(a_i)$ maps to zero in $\Omega^1_{K/k}$.  As $K/k$ is geometrically regular, the Cartier isomorphism implies that $a_i$ is a $p$-th power in $K$; as $X$ was normal, this implies that it is a $p$-th power in $A$.  Let $b_i$ be a $p$-th root of $a_i$ (which is unique by the reducedness of $A$).  Then the map $\widetilde f: k[t_1, \cdots t_n]\to A$ given by $$\widetilde f: t_i\mapsto b_i$$ kills the ideal $I$ (again by the reducedness of $X$), and hence gives a map $\bar f: B\to A$ so that $f=\on{Frob}_p\circ \bar f.$
\end{proof}
We will also need a lemma estimating the f-amplitude of a vector bundle in terms of the f-amplitude of its restriction to an ample divisor.  This is essentially \cite[Lemma 6.1]{keeler}, but we remove the extraneous hypothesis that the divisor be very ample.
\begin{lem}\label{ampledivisorbound}
Let $X$ be a projective variety and let $D\subset X$ be an ample Cartier divisor.  Then if $\mcl{E}$ is a vector bundle on $X$, $$\phi(\mcl{E})\leq \phi(\mcl{E}|_D)+1.$$
\end{lem}
\begin{proof}
This is proven in \cite[Lemma 6.1]{keeler} for $D$ very ample, so we reduce to that case.  It suffices to work in positive characteristic.  For $k\gg0$, the divisor $D_{p^k}$ (defined by the $p^k$-th power of the ideal sheaf of $D$) is very ample.  Let $\tau: D_{p^k}\to D$ be the morphism defined in Lemma \ref{frobeniusfactors}.  Then $$\tau^*(\mcl{E}|_D)=\mcl{E}^{(p^k)}|_{D_{p^k}}$$ by the construction of $\tau$.  Thus by \cite[Lemma 6.1]{keeler}, we have $$\phi(\mcl{E})=\phi(\mcl{E}^{(p^k)})\leq \phi(\mcl{E}^{(p^k)}|_{D_{p^k}})+1=\phi(\tau^*(\mcl{E}|_D))+1\leq\phi(\mcl{E}|_D)+1$$ where we use the finiteness of $\tau$ and \cite[Theorem 2.5(4)]{arapura-f-amplitude} to obtain the last inequality.
\end{proof}
Finally, we need a lemma on the uniqueness of factorizing a morphism through Frobenius.
\begin{lem}\label{uniquefactorizationfrobenius}
Let $k$ be a perfect field of characteristic $p>0$.  Let $X$ be a smooth projective $k$-variety with $2<\dim(X)<p$, and $D\subset X$ an ample divisor.  Suppose that $X$ lifts to $W_2(k)$.  Let $Y$ be a scheme and $f: D\to Y$ a morphism.  Then there is at most one morphism $g: D\to Y$ so that $f=g\circ \on{Frob}_p$.
\end{lem}
\begin{proof}
Let $\sqrt[p]{0}\subset \mcl{O}_D$ be the ideal sheaf $$\sqrt[p]{0}=\{f\mid f^p=0\}.$$  Then for any $g_1, g_2$ such that $$g_1\circ \on{Frob}_p=g_2\circ \on{Frob}_p=f$$ we have that $$g_1-g_2\in \on{Hom}_{f^{-1}\mcl{O}_Y}(f^{-1}\mcl{O}_Y, \sqrt[p]{0})\subset \Gamma(D, \sqrt[p]{0}).$$  Thus it suffices to show that $\Gamma(D, \sqrt[p]{0})=0$.  As $\Gamma(D, \sqrt[p]{0})$ is a nilpotent ideal in $\Gamma(D, \mcl{O}_D)$, it suffices to show that $\Gamma(D, \mcl{O}_D)$ is reduced.  From the short exact sequence $$0\to \mcl{O}_X(-D)\to \mcl{O}_X\to \mcl{O}_D\to 0$$ and the fact that $\Gamma(X, \mcl{O}_X)$ is a product of fields, it is enough to show that $$H^1(X, \mcl{O}_X(-D))=0.$$  But this is immediate from Theorem \ref{arapuravanishing}.
\end{proof}
We are now ready to give a positive characteristic version of Theorem \ref{charzeroextensionthm} from which we will deduce the result in characterstic zero.
\begin{thm}\label{charpliftextensionthm}
Let $L$ be a perfect field of characteristic $p>0$.  Let $X$ be a smooth $L$-variety and $D\subset X$ an ample Cartier divisor, with $\dim(X)\geq 3$, such that $X$ lifts to $W_2(L)$, and such that $\dim(X)<p$.  Let $Y$ be a smooth $k$-variety, and let $f: D\to Y$ be a morphism. Suppose that  $$\phi(\mcl{N}_{D/X}\otimes f^*\Omega^1_Y)<\dim(D)-1.$$ Suppose further that  
\begin{enumerate}
\item $\dim(Y)<\dim(D)$, or
\item $Y$ is projective and $\dim(\on{im}(f))<\dim(D)-1$, or
\item $Y_{\bar L}$ contains no rational curves.
\end{enumerate}   
Then $f$ extends uniquely to a morphism $X\to Y$.  
\end{thm} 
\begin{proof}
We first prove existence of an extension.

By Theorem \ref{frobeniusvanishing}, there exists $k\geq 0$ so that $F_{Y/L}^k\circ f$ extends to a morphism $\widehat D\to Y^{(p^k)}$.  By Corollary \ref{sectionalgebraization}, there exists a Zariski open set $U\subset X$, with $D\subset U$, so that this map extends to a morphism $U\to Y^{(p^k)}$.  Finally, by Corollary \ref{nonproperextension} in Case (1), Corollary \ref{smallimageextension} in Case (2), or Proposition \ref{noratlcurvesextension} in Case (3) this map extends to a map $\widetilde f: X\to Y^{(p^k)}$.  We wish to show that if $k>0$, $\widetilde f$ factors through $F_{Y/L}$, or equivalently that the induced map $\bar f: X\to Y$ factors through $\on{Frob}_p$.  

Suppose that $k>0$.  By Lemma \ref{factorsthroughfrobenius}, it suffices to show that the map $\bar f^*\Omega^1_{Y}\to \Omega^1_X$ is identically zero.  Recall that we have two  exact sequences 
\begin{equation*}
0\to \Omega^1_X(-p^kD)\to \Omega^1_X\to \Omega^1_X|_{D_{p^k}}\to 0
\end{equation*}
and 
\begin{equation*}
 \mcl{N}_{D_{p^k}/X}^\vee\to \Omega^1_X|_{D_{p^k}}\to \Omega^1_{D_{p^k}}\to 0.
\end{equation*}
A local computation shows that the map $$\mcl{N}_{D_{p^k}/X}^\vee\to \Omega^1_X|_{D_{p^k}}$$ is identically zero, so $\Omega^1_X|_{D_{p^k}}= \Omega^1_{D_{p^k}}$.  Thus
\begin{equation}\label{funses}
 \Omega^1_X(-p^kD)= \ker(\Omega^1_X\to \Omega^1_{D_{p^k}})
\end{equation}
Because $k>0$, the composite map $$\bar f^*\Omega^1_{Y}\to \Omega^1_X\to \Omega^1_{D_{p^k}}$$ is zero (see Remark \ref{zeroonDpk}),  so we must show that the induced map $$\bar f^*\Omega^1_{Y}\to \ker(\Omega^1_X\to \Omega^1_{D_{p^k}})=\Omega^1_X(-p^kD)$$ is zero, where the equality follows from Equation \ref{funses}.  It suffices to show that $$\on{Hom}(\bar f^*\Omega^1_Y, \Omega^1_X(-p^kD))=0.$$
Now, $$\bar f^* \Omega^1_Y(p^kD)|_{D}=\on{Frob}_p^{k*}( f^*\Omega^1_Y\otimes \mcl{N}_{D/X}).$$ Thus $$\phi(\bar f^* \Omega^1_Y(p^kD)|_{D})< \dim(D)-1$$ by assumption, and Lemma \ref{ampledivisorbound} implies that $$\phi(\bar f^* \Omega^1_Y(p^kD))<\dim(X)-1.$$  Hence
\begin{align*}
\on{Hom}(\bar f^*\Omega^1_Y, \Omega^1_X(-p^kD)) &=H^0(X, \Omega^1_X\otimes (\bar f^*\Omega^1_Y(p^kD))^\vee)\\
&=H^{\dim(X)}(X, \Omega^{\dim(X)-1}_X\otimes  (\bar f^*\Omega^1_Y(p^kD)))^\vee\\
&=0
\end{align*} by Serre duality and Theorem \ref{arapuravanishing}, as desired.  So we have $\bar{f}=\on{Frob}_p\circ h$ for some $h$; we wish to show that $h|_D=\on{Frob}_p^{k-1}\circ f$.  This follows immediately from Lemma \ref{uniquefactorizationfrobenius}, however.  (If $D$ was assumed to be reduced, this lemma would have been unnecessary, as $\on{Frob}_p: D\to D$ would be an epimorphism.)  Thus we have shown existence of an extension, by induction on $k$.

We now prove uniqueness of the extension.  By Theorem \ref{unique-extensions}, it is enough to show that for our given extension $g$, we have $$\phi(f^*\Omega^1_Y(D))<\dim(D).$$ But $f^*\Omega^1(D)|_D=\mcl{N}_{D/X}\otimes f^*\Omega^1_Y$, so were are done by Lemma \ref{ampledivisorbound}.
\end{proof}
\begin{rem}
Observe that in Theorem \ref{charpliftextensionthm}, we require that the variety $X$ lift to $W_2(L)$, but there is no such requirement on $f, D,$ or $Y$.
\end{rem}
Finally, we may prove Theorem \ref{charzeroextensionthm}.
\begin{proof}[Proof of Theorem \ref{charzeroextensionthm}]
We may spread $(X, D, Y, f)$ out over a finite type $\mbb{Z}$-scheme $S$; after shrinking $S$, we may assume that for each closed point $s\in S$, the hypotheses of Theorem \ref{charpliftextensionthm} are satisfied.  Thus each $f_s$ extends uniquely to a map $X_s\to Y_s$.  This completes the proof, as the $S$-scheme parametrizing extensions of $f$ to $X$ is of finite type over $\mbb{Z}$.
\end{proof}
The main case in which we will be able to check the hypotheses of Theorem \ref{charzeroextensionthm} is when $\Omega^1_Y$ is nef.  For this application and examples of $Y$ with nef cotangent bundle, see Section \ref{nefcotangentbundle}.

While we omit the proof, essentially identical arguments give the following ``section" version of Theorem \ref{charzeroextensionthm}:
\begin{thm}\label{charzerosectionextensionthm}
Let $k$ be a field of characteristic zero.  Let $X$ be a smooth $k$-variety and $D\subset X$ an ample Cartier divisor, with $\dim(X)\geq 3$.  Let $f: Y\to X$ be a smooth proper morphism $s: D\to Y$ be a section to $f_D$, such that $$\phi(\mcl{N}_{D/X}\otimes s^*\Omega^1_{Y/X})<\dim(D)-1.$$   Suppose further that  
\begin{enumerate}
\item $\on{rel. dim.}(Y/X)<\dim(D)$, or
\item $f$ admits a model which is finite type over $\mbb{Z}$, so that the geometric fibers of this model contains no rational curves.
\end{enumerate}   
Then $s$ extends uniquely to a section to $f$.
\end{thm}
There is also an analogue of Theorem \ref{charzerosectionextensionthm} for schemes over a perfect field $L$ of positive characteristic which lift to $W_2(L)$, which we will not use.  The proof is essentially identical to that of Theorem \ref{charpliftextensionthm}, though it is notationally more complicated.

\section{Lefschetz for maps to Deligne-Mumford stacks}
Let us briefly recall our main goal.  Let $X$ be a smooth projective variety an $D\subset X$ an ample Cartier divisor.  Then we wish to study the problem of extending a map $D\to Y$ to a map $X\to Y$.  As we will see in Chapter \ref{applications}, the results of Section \ref{charzeroextension} provide a reasonably satisfactory answer to this problem when $Y$ is a smooth scheme.  The purpose of this chapter is to extend some of the results of Section \ref{charzeroextension} to the case of maps whose target is a (reasonable) Deligne-Mumford stack.  

Many of the results of the previous few chapters hold with few changes when $Y$ is an Artin stack with finite diagonal; however, the proofs are somewhat complicated, and the results in this chapter will suffice for all of the applications we have in mind.  We will restrict our attention to the case where $Y$ is a smooth Deligne-Mumford stack.  This covers many cases of interest, e.g. $BG$ with $G$ a finite \'etale group scheme, the moduli space of curves $\mcl{M}_g$, the moduli space of principally polarized Abelian varieties $\mcl{A}_g$ as well as all other Shimura varieties, etc.

Our main result is the following improvement of Theorem \ref{charzeroextensionthm}:
\begin{thm}\label{charzeroextensionthmdmstacks}
Let $k$ be a field of characteristic zero.  Let $X$ be a smooth $k$-variety and $D\subset X$ an ample Cartier divisor, with $\dim(X)\geq 3$.  Let $\mcl{Y}$ be a smooth Deligne-Mumford stack over $k$.  Let $f: D\to \mcl{Y}$ be a morphism such that $$\phi(\mcl{N}_{D/X}\otimes f^*\Omega^1_Y)<\dim(D)-1.$$   Suppose further that  
\begin{enumerate}
\item $\dim(\mcl{Y})<\dim(D)$, or
\item $\dim(f(D))\leq \dim(D)-2$ and the coarse space of $\mcl{Y}$ is projective, or
\item $\mcl{Y}$ is proper and admits a model which is finite type over $\mbb{Z}$ and whose geometric fibers contain no rational curves (i.e. any map from a rational curve to this model is constant) .
\end{enumerate}   
Then $f$ extends uniquely to a morphism $X\to \mcl{Y}$.
\end{thm}
\begin{proof}
We sketch how to extend the results in previous chapters to the case where $\mcl{Y}$ is a smooth Deligne-Mumford stack.  We first give an analogue of Corollary \ref{sectionalgebraization}.
\begin{lem}
Let $X, D, \mcl{Y}$ be as in the theorem, and let $\widehat{D}$ be the formal scheme obtained by completing $X$ at $D$.  Then any map $f: \widehat D\to \mcl{Y}$ extends to a Zariski-open neighborhood of $D$.
\end{lem}
\begin{proof}
By the main theorem of \cite{olsson_proper}, there exists a quasi-projective scheme $Y'$ and a proper $1$-morphism $Y'\to \mcl{Y}$.  Let $Y''=Y'\times X$, and let $Y''_D$ be the base change of $Y''$ along the inclusion $D\hookrightarrow X$.  Let $\widehat{Y''_D}$ be the formal scheme obtained by completing $Y''$ at $Y''_D$, and let $Z\subset \widehat{Y''_D}$ be the closed formal subscheme given by taking the preimage of the graph of $f$ (the scheme-theoretic image of the map $(\on{id}, f): \widehat{D}\to \widehat{D}\times \mcl{Y}$).  By Corollary \ref{subschemealgebraization}, $Z$ extends to a subscheme $Z'$ of $Y''$.  

Now let $U\to \mcl{Y}$ be an \'etale cover of $\mcl{Y}$, and consider $$r_1: U\times_\mcl{Y} Y''\to U\times X,$$  $$\pi_1: U\times_\mcl{Y} Y''\to Y'',$$ $$r_2: U\times_\mcl{Y} U \times_\mcl{Y} Y''\to U\times_\mcl{Y} U \times X,$$ and $$\pi_2: U\times_\mcl{Y} U\times_\mcl{Y} Y''\to Y''.$$  Then $r_1(\pi_1^{-1}(Z'))\subset U\times X$ and $r_2(\pi_2^{-1}(Z'))$ are \'etale over a Zariski-open neighborhood of $D$ in $X$; furthermore, over a neighborhood of $D$, the natural map $$r_2(\pi_2^{-1}(Z'))\to r_1(\pi_1^{-1}(Z'))\times_X r_1(\pi_1^{-1}(Z'))$$ is an isomorphism.  Thus the maps $$r_1(\pi_1^{-1}(Z'))\to U$$ and $$r_2(\pi_2^{-1}(Z'))\to U\times_\mcl{Y} U$$ provide descent data for a map from a Zariski-open neighborhood of $D$ to $\mcl{Y}$.
\end{proof}
\begin{rem}
The argument above works for Artin stacks with finite diagonal.
\end{rem}
We now give an analogue of Corollary \ref{nonproperextension} and Proposition \ref{noratlcurvesextension}.
\begin{lem}
Let $X, D, \mcl{Y}$ be as in the theorem, and let $U\subset X$ be a Zariski-open containing $D$.  Then the restriction map $\mcl{Y}(X)\to \mcl{Y}(U)$ is an equivalence.
\end{lem}
\begin{proof}
Let $Y$ be the coarse space of $\mcl{Y}$, which exists by the Keel-Mori theorem \cite{keel-mori}.  Let $f: U\to \mcl{Y}$ be a map; we may resolve the induced map $U\to Y$ to obtain a scheme $X'$, proper over $X$ and a map $f': X'\to Y$.  Let $V\to \mcl{Y}$ be an \'etale cover and let $X^{(2)}\subset X\times V$ be the scheme-theoretic image of the natural map $$X'\times_Y V\to X\times V.$$  By properness of $X'$, $X^{(2)}$ is a cover of $X$.   Likewise, let $X^{(3)}\subset X\times V\times V$ be the scheme-theoretic image of the natural map $$X'\times_Y V\times_Y V\to X\times V\times_{\mcl{Y}} V.$$  We claim that the maps $X^{(2)}\to V, X^{(3)}\to V\times V$ give descent data for a map $X\to \mcl{Y}$.

Indeed, in cases (1) and (2) we may take $X'=X$, by Corollaries \ref{nonproperextension} and \ref{smallimageextension} respectively.  Then $X^{(2)}$ and $X^{(3)}$ are quasi-finite over $X$, and hence are \'etale over $X$ by Zariski-Nagata purity.  The natural map $X^{(2)}\times_X X^{(2)}\to X^{(3)}$ is an isomorphism over $U$ and hence an isomorphism over all of $X$, using the \'etaleness from the previous sentence.

In case (3), it again suffices to show that $X^{(2)}$ and $X^{(3)}$ are quasi-finite over $X$.  Indeed, it is enough to show that the maps $$X'\times_Y V\to X\times V,$$ $$X'\times_Y V \times_Y V\to X\times V\times_{\mcl{Y}} V$$ contract the preimages of the exceptional divisor of $X'\to X$.  But these exceptional divisors are ruled; by assumption $V$ and $V \times_{\mcl{Y}} V$ contain no rational curves, so the proof is complete.
\end{proof}
Finally, we need an analogue of Lemma \ref{factorsthroughfrobenius}.
\begin{lem}\label{factorsthroughfrobeniusdmstacks}
Let $k$ be a perfect field of characteristic $p>0$, and let $X$ be a normal, reduced $k$-variety.  Let $\mcl{Y}$ be an arbitrary $k$-Deligne-Mumford stack.  Then a morphism $f: X\to Y$ factors through $\on{Frob}_p: Y\to Y$ if and only if the induced map $f^*\Omega^1_Y\to \Omega^1_X$ is zero; furthermore, this factorization is unique up to canonical isomorphism.
\end{lem}
\begin{proof}
Let $U\to \mcl{Y}$ be an \'etale cover.  Let $X^{(2)}=X\times_\mcl{Y} U$ and $X^{(3)}=X\times_\mcl{Y} U\times_\mcl{Y} U$; let $$f^{(2)}: X^{(2)}\to U$$ and $$f^{(3)}: X^{(3)}\to U\times_\mcl{Y} U$$ be the descent data for $f$.  Then by Lemma \ref{factorsthroughfrobenius}, $f^{(2)}, f^{(3)}$ factor uniquely through Frobenius.  These factorizations (and their uniqueness) provide descent data for a (unique up to canonical isomorphism) factorization of $f$ through Frobenius.
\end{proof}
Now we may exactly follow the proof of Theorem \ref{charzeroextensionthm}, replacing the corresponding scheme-theoretic lemmas with the lemmas above.
\end{proof}
Essentially identical arguments give an analogue of Theorem \ref{charpliftextensionthm}, etc.

We now give an analogue of Theorem \ref{unique-extensions}.
\begin{thm}\label{charzerouniquenessthmdmstacks}
Let $k$ be a field, and let $X$ be a smooth projective $k$-variety with $\dim(X)\geq 2$.  Let $D\subset X$ be an ample divisor, and let $\mcl{Y}$ be a smooth $k$-Deligne-Mumford stack.  Let $f: X\to \mcl{Y}$ be a morphism such that $$\phi(f^*\Omega^1_Y(D))<\dim(D).$$ Then if 
\begin{enumerate}
\item $\on{char}(k)=0$, or
\item $k$ is perfect of characteristic $p>\dim(X)$ and $X$ lifts to $W_2(k)$, 
\end{enumerate}
any extension of $f|_D: D\to Y$ to a map $X\to Y$, is canonically isomorphic to $f$ itself.
\end{thm}
\begin{proof}
The proof of Theorem \ref{unique-extensions} shows that $f$ admits a unique (up to canonical isomorphism) extension to $\widehat D$.  Now we may imitate the argument of Corollary \ref{sectionsareequal} to conclude the result.
\end{proof}

\section{Applications}\label{applications}
We now come to the applications of the results in previous chapters.  The main work in this section is to identify positivity properties of the cotangent bundle of a smooth variety which allow one to verify the hypotheses of Theorems \ref{charzerosectionextensionthm} and \ref{charzeroextensionthmdmstacks}.  We also give several sporadic examples of varieties (and Deligne-Mumford stacks) such that maps into them from an ample divisor in $X$ automatically extend to maps from $X$.  Of particular interest is the case where $Y$ represents a natural moduli functor (e.g. $\mcl{M}_g$).
\subsection{Maps to varieties with nef cotangent bundle}\label{nefcotangentbundle}
The first case of interest is that of smooth proper varieties with nef cotangent bundle.  Our main result is:
\begin{thm}\label{nefcotangentbundleextension}
Let $k$ be a field of characteristic zero.  Let $X$ be a smooth projective $k$-variety with $\dim(X)\geq 3$, and let $D\subset X$ be an ample Cartier divisor.  Let $Y$ be a smooth Deligne-Mumford stack over $k$ and let $f: D\to Y$ be a morphism.  Suppose that $\Omega^1_Y$ is nef, that there exists a scheme $Y'$ and a finite surjective \'etale morphism $Y'\to Y$, and that  $\dim(Y)<\dim(D)$.  Then $f$ extends uniquely to a morphism $X\to Y$.  If $\dim(Y)=\dim(D)$, at most one such extension exists.
\end{thm}
\begin{proof}
We first deal with the case $\dim(Y)<\dim(D)$.  We must verify the hypotheses of Theorems \ref{charzeroextensionthmdmstacks}.  The only non-obvious hypothesis is that $$\phi(f^*\Omega^1_Y\otimes \mcl{N}_{D/X})<\dim(D)-1.$$  But $f^*\Omega^1_Y\otimes \mcl{N}_{D/X}$ is ample, as $f^*\Omega^1_Y$ is nef and $\mcl{N}_{D/X}=\mcl{O}_X(D)|_D$ is ample.  So by Theorem \ref{arapurabound}, we have $$\phi(f^*\Omega^1_Y\otimes \mcl{N}_{D/X})<\on{rk}(f^*\Omega^1_Y\otimes \mcl{N}_{D/X})\leq\dim(D)-1,$$ as desired.

Now we deal with the case $\dim(Y)=\dim(D)$.  We must verify the hypotheses of Theorem \ref{charzerouniquenessthmdmstacks}.  The only non-obvious hypothesis is that if an extension $g$ exists, we have $$\phi(g^*\Omega^1_Y(D))<\dim(D).$$  But this follows exactly as before.
\end{proof}
\begin{thm}\label{nefsectionsextend}
Let $k$ be a field of characteristic zero.  Let $X$ be a smooth projective $k$-variety with $\dim(X)\geq 3$, and $D\subset X$ an ample Cartier divisor.  Let $f:Y\to X$ be a smooth proper morphism, and let $s: D\to Y$ be a section to $f_D$.  Suppose that $\Omega^1_{Y/X}$ is nef, and that $\on{rel. dim.}(f)<\dim(D)$.  Then $s$ extends uniquely to a section to $f$.
\end{thm}
\begin{proof}
We must verify the hypotheses of Theorem \ref{charzerosectionextensionthm}.  As before, the only non-obvious hypothesis is that $$\phi(s^*\Omega^1_{Y/X}\otimes \mcl{N}_{D/X})<\dim(D)-1.$$ But $s^*\Omega^1_{Y/X}\otimes \mcl{N}_{D/X}$ is ample, as $s^*\Omega^1_{Y/X}$ is nef and $\mcl{N}_{D/X}=\mcl{O}_X(D)|_D$ is ample.  So by Theorem \ref{arapurabound}, we have $$\phi(f^*\Omega^1_{Y/X}\otimes \mcl{N}_{D/X})<\on{rk}(f^*\Omega^1_{Y/X}\otimes \mcl{N}_{D/X})\leq\dim(D)-1,$$ as desired.
\end{proof}
We now collect some examples of varieties with nef cotangent bundle, to which the two theorems above apply.
\begin{thm}\label{sourceofnefexamples}
Let $Y$ be a smooth variety such that one of the following holds:
\begin{enumerate}
\item $Y$ is a curve of genus at least $1$.
\item $Y$ admits an unramified map to an abelian variety.
\item $\on{Sym}^n\Omega^1_Y$ is globally generated for some $n> 0$.
\item $Y$ is a compact quotient of a bounded domain in $\mbb{C}^n$ or in a Stein manifold.
\item $Y$ is a closed subvariety of a smooth variety with nef cotangent bundle.
\item $Y$ admits a smooth map $f: Y\to X$, where $X$ is smooth and both $\Omega^1_X, \Omega^1_{Y/X}$ are nef.
\item There exists $Y'$ and a surjective \'etale morphism $Y'\to Y$ or $Y\to Y'$, and $\Omega^1_{Y'}$ is nef.
\end{enumerate}
Then $Y$ has nef cotangent bundle.
\end{thm}
\begin{proof}
\begin{enumerate}
\item Curves of genus at least $1$ have globally generated cotangent bundles and admit unramified maps to Abelian varieties (in fact, for smooth varieties, these two conditions are equivalent), so the result follows from either parts (2) or (3).
\item Let $f: Y\to A$ be an unramified map to an Abelian variety, and choose a trivialization of $\Omega^1_A$.  Then the natural map $\mcl{O}_Y^{\dim(A)}\simeq f^*\Omega^1_A\to \Omega^1_Y$ is surjective, so $\Omega^1_Y$ is globally generated.  Hence the result follows from part (3).
\item A vector bundle is nef if and only if some symmetric power is nef; this is \cite[Theorem 6.2.12(iii)]{lazarsfeld2004positivity2}.  Globally generated vector bundles are nef, so the result follows.
\item This is \cite[Theorem 6]{kratz} and the remarks following it.
\item Let $\iota: Y\hookrightarrow X$ be a closed embedding, where $\Omega^1_X$ is nef.  Then the map $\iota^*\Omega^1_X\to \Omega^1_Y$ is surjective.  But quotients of nef vector bundles are nef, by \cite[Theorem 6.2.12(i)]{lazarsfeld2004positivity2}.
\item  There is a short exact sequence $$0\to f^*\Omega^1_X\to \Omega^1_Y\to \Omega^1_{Y/X}\to 0.$$   But extensions of nef vector bundles are nef, by \cite[Theorem 6.2.12(ii)]{lazarsfeld2004positivity2}, giving the result.
\item There is an isomorphism $f^*\Omega^1_Y\simeq \Omega^1_{Y'}$, so if $\Omega^1_Y$ is nef, so is $\Omega^1_{Y'}$.  The converse follows from \cite[Proposition 6.1.7(iv)]{lazarsfeld2004positivity2}.
\end{enumerate}
\end{proof}
\begin{cor}
Let $X\to Y$ be a smooth proper relative curve of genus $\geq 1$, and suppose that $Y$ has nef cotangent bundle.  Then $X$ has nef cotangent bundle.
\end{cor}
\begin{proof}
By Theorem \ref{sourceofnefexamples}(6), it suffices to show that $\Omega^1_{X/Y}$ is nef.  But this is well-known; see e.g. \cite[Theorem 0.4]{keel}.
\end{proof}
Good sources of varieties with nef cotangent bundle include \cite{brotbek}, \cite{debarre2}, \cite{horing}, \cite{jabbusch}, \cite{kratz}, \cite{spurr}.

We also give some examples of classes of morphisms with nef \emph{relative} cotangent bundle, where one may apply Theorem \ref{nefsectionsextend}.
\begin{thm}\label{relativelygloballygeneratedsections}
Let $k$ be a field of characteristic zero.  Let $f: Y\to X$ be a smooth morphism of smooth $k$-varieties with $\Omega^1_{X/Y}$ relatively globally generated (i.e. the map $$f^*f_*\Omega^1_{Y/X}\to \Omega^1_{Y/X}$$ is surjective.)  Then $\Omega^1_{Y/X}$ is nef.
\end{thm}
\begin{proof}
As the quotient of a nef vector bundle is nef, it suffices to show that $f^*f_*\Omega^1_{Y/X}$ is nef; as nefness is preserved by pullbacks, it is enough to show that $f_*\Omega^1_{Y/X}$ is nef.  But this is a consequence of Griffiths positivity; see e.g. \cite[Theorem 5]{kratz}, or \cite[Corollary 7.8]{griffiths} for the dual result.
\end{proof}
\begin{cor}\label{sectionswithglobalgeneration}
Let $k$ be a field of characteristic zero.  Let $f: Y\to X$ be a smooth proper morphism, with $X$ a smooth projective $k$-variety, and $\dim(X)\geq 3$.  Let $D\subset X$ be an ample Cartier divisor, and suppose $\on{rel. dim.}(f)<\dim(D)$.  Then if each geometric fiber of $f$ has globally generated cotangent bundle, the natural map $$\on{Sections}(f)\to \on{Sections}(f_D)$$ is an isomorphism.
\end{cor}
\begin{proof}
By theorem \ref{relativelygloballygeneratedsections}, it suffices to check that $\Omega^1_{Y/X}$ is relatively globally generated.  But the formation of $f_*\Omega^1_{X/Y}$ commutes with base change, by \cite[Section 4]{deligne-illusie}, so this is true on geometric fibers, and the result follows.
\end{proof}
\begin{rem}
In particular, relative curves of genus at least one and (torsors for) Abelian schemes satisfy the hypotheses of Theorem \ref{relativelygloballygeneratedsections}.
\end{rem}
\begin{rem}
There are non-isotrivial families of curves and Abelian varieties over proper bases.  Indeed, both $\mcl{M}_g$ and $\mcl{A}_g$ admit compactifications whose boundaries have codimension $\geq 2$, so a general complete intersection curve will miss the boundary.
\end{rem}
\subsection{Moduli spaces with nef cotangent bundle}
We now give several examples of moduli spaces $\mcl{M}$ with nef cotangent bundle.  Applying Theorem \ref{nefcotangentbundle}, we will see that a family of objects parametrized by such a moduli space $\mcl{M}$ over an ample divisor $D\subset X$ will extend uniquely to $X$, as long as $\dim(\mcl{M})<\dim(D)$.
\begin{lem}\label{trivialcotangentnefness}
Let $f: Y\to X$ be a smooth projective morphism over a field of characteristic zero, of relative dimension $n$.  Suppose that each geometric fiber of $f$ has trivial cotangent bundle.  Then $(\mbf{R}^1f_*T_{Y/X})^\vee$ is a nef vector bundle.
\end{lem}
\begin{proof}
We have $$(\mbf{R}^1f_*T^1_{Y/X})^\vee=\mbf{R}^1f_*(\Omega^1_{Y/X}\otimes \omega_{Y/X}).$$  As each geometric fiber of $f$ has globally generated cotangent bundle, $\Omega^1_{Y/X}$ is isomorphic to $f^*f_*\Omega^1_{X/Y}$, arguing as in Corollary \ref{sectionswithglobalgeneration}.  Now, $$\mbf{R}^1f_*(f^*f_*\Omega^1_{Y/X}\otimes \omega_{Y/X})=f_*\Omega^1_{Y/X}\otimes \mbf{R}^1f_*\omega_{Y/X}=f_*\Omega^1_{Y/X}\otimes(\mbf{R}^{n-1}f_*\mcl{O}_{Y/X})^\vee.$$
But both of these bundles $f_*\Omega^1_{Y/X}$ and $\mbf{R}^{n-1}f_*\mcl{O}_{Y/X}$ are nef vector bundles.  That they are vector bundles follows from \cite[Section 4]{deligne-illusie} (or if $X$ is reduced, simply from the local constancy of Hodge numbers); that they are nef follows from Griffiths positivity (see e.g. \cite[Theorem 5]{kratz}, or \cite[Corollary 7.8]{griffiths}), so their tensor product is as well.
\end{proof}
\begin{cor}
The cotangent bundle of $\mcl{A}_g$ is nef over a field of characterisitc zero.
\end{cor}
\begin{proof}
Let $A\to \mcl{A}_g$ be the universal family.  Then the cotangent bundle of $\mcl{A}_g$ is a quotient of $(\mbf{R}^1f_*T_{Y/X})^\vee$, and is hence nef by Lemma \ref{trivialcotangentnefness}.
\end{proof}
\begin{rem}
This result is false in every positive characteristic, for $g\geq 2$.  This follows immediately from Moret-Bailly's construction of complete rational curves in $\mcl{A}_g$, \cite{Moret-Bailly}.
\end{rem}
\begin{lem}\label{torellinef}
Let $f: Y\to X$ be a smooth projective morphism over a field of characteristic zero, such that the natural map $$\mbf{R}^1f_*T_{Y/X}\to (f_*\Omega^1_{Y/X})^\vee\otimes (\mbf{R}^1f_*\mcl{O}_{Y/X})$$ is injective (that is, the family $f$ satisfies an infinitesimal Torelli theorem in weight one).  Then if $\mcl{M}$ is a smooth moduli space of polarized varieties such that $f$ is induced by a map $g: X\to \mcl{M}$, the vector bundle $f^*\Omega^1_{\mcl{M}}$ is nef.
\end{lem}
\begin{proof}
Let $a: A\to X$ be the Albanese $X$-scheme associated to $f$.  Then $(f_*\Omega^1_{Y/X})^\vee\otimes (\mbf{R}^1f_*\mcl{O}_{Y/X})$ is isomorphic to $\mbf{R}^1a_*T_{A/X}$, and hence is nef by Lemma \ref{trivialcotangentnefness} (one may also see this directly using Griffiths positivity).  But $f^*\Omega^1_{\mcl{M}}$ is a quotient of $(\mbf{R}^1f_*T_{Y/X})^\vee$ and hence a quotient of $f_*\Omega^1_{Y/X})^\vee\otimes (\mbf{R}^1f_*\mcl{O}_{Y/X})$ by assumption, and hence is nef as well.
\end{proof}
\begin{cor}\label{mgnefcotangentbundle}
Let $g\geq 2, n\geq 0$.  Then over a field of characteristic zero, $\Omega^1_{\mcl{M}_{g, n}}$ is nef.  The cotangent bundle of $\mcl{M}_{1, 1+n}$ is nef as well.
\end{cor}
\begin{proof}
The universal families $\mcl{M}_{g}$ and $\mcl{M}_{1, 1}$ both satisfy the hypotheses of Lemma \ref{torellinef}, so we proceed by induction on $n$.  But the map $\mcl{M}_{g,n}\to \mcl{M}_{g, n-1}$ exhibiting $\mcl{M}_{g,n}$ as the universal curve over $\mcl{M}_g$ has relatively globally generated relative cotangent bundle, hence by Theorem \ref{sourceofnefexamples}(6), the induction step is complete.
\end{proof}
\begin{rem}
This corollary holds true in positive characteristic as well; see \cite[Theorem 4.3]{kollar}.
\end{rem}
\begin{rem}
The cotangent bundle of $\overline{\mathcal{M}_{g, n}}$ is \emph{not} nef; however, the logarithmic cotangent bundle $\Omega^1_{\overline{\mcl{M}_{g, n}}}(\log \delta)$, for $\delta:=\overline{\mathcal{M}_{g, n}}\setminus \mathcal{M}_{g, n}$ is nef. One may prove logarithmic versions of Theorems \ref{maintheorem} and \ref{maintheorem2}, and this observation about the nefness allows one to prove a Lefschetz theorem for stable curves.  
\end{rem}
\begin{thm}\label{hyperkahlernef}
Let $\mcl{M}$ be a moduli space parametrizing polarized hyperk\"ahler manifolds.  Then $\Omega^1_{\mcl{M}}$ is nef.
\end{thm}
\begin{proof}
Let $f: Y\to X$ be a family of polarized hyperk\"ahler manifolds, with induced classifying morphism $g: X\to \mcl{M}$.  We wish to show that $f^*\Omega^1_{\mcl{M}}$ is nef.  Let $A\to X$ be the Kuga-Satake Abelian scheme associated to $f$.  We may consider the natural map $$f^*T^1_{\mcl{M}}\to \mbf{R}^1f_*T_{Y/X}\to \mbf{R}^1f_*T_{A/Y}\to  (f_*\Omega^1_{A/X})^\vee\otimes (\mbf{R}^1f_*\mcl{O}_{A/X})\to  (f_*\Omega^2_{Y/X})^\vee\otimes (\mbf{R}^1f_*\Omega^1_{Y/X});$$ this map is injective by the local torelli Theorem for hyperk\"ahlers.  Thus the map $$f^*T^1_{\mcl{M}}\to \mbf{R}^1f_*T_{Y/X}\to  (f_*\Omega^1_{A/X})^\vee\otimes (\mbf{R}^1f_*\mcl{O}_{A/X})$$ is injective.  But this last vector bundle has nef dual; hence $f^*\Omega^1_{\mcl{M}}$ is nef as desired.
\end{proof}
\begin{thm}\label{calabiyaunef}
Let $\mcl{M}$ be a moduli space parametrizing smooth polarized varieties $X$ with trivial canonical bundle and $h^{2,0}=0$.  Then $\Omega^1_{\mcl{M}}$ is nef.  
\end{thm}
\begin{proof}
It suffices to show that if $f: Y\to X$ is a family of smooth polarized varieties with trivial canonical bundle and $h^{2,0}=0$, then $$\mbf{R}^1f_*(\Omega^1_{Y/X}\otimes \omega_{Y/X})$$ is nef.  As $\omega_{Y/X}$ is trivial on geometric fibers of $f$, we have that $\omega_{Y/X}=f^*f_*\omega_{Y/X}$, so $$\mbf{R}^1f_*(\Omega^1_{Y/X}\otimes \omega_{Y/X})=\mbf{R}^1f_*(\Omega^1_{Y/X}\otimes f^*f_*\omega_{Y/X})=\mbf{R}^1f_*\Omega^1_{Y/X}\otimes f_*\omega_{Y/X}$$  But both $\omega_{Y/X}$ and $\mbf{R}^1f_*\Omega^1_{Y/X}$ are nef by Griffiths positivity \cite[Corollary 7.8]{griffiths}, where the latter fact follows from the fact that $h^{2,0}=0$ and the fact that $\mbf{R}^1f_*\Omega^1_{Y/X}$ is self-dual.
\end{proof}
\begin{rem}
Moduli spaces of odd dimensional Calabi-Yau varieties are quasi-affine \cite{todorov}, so in that case the theorem above is vacuous.  On the other hand, there are moduli spaces of polarized $2n$-dimensional Calabi-Yau varieties containing non-trivial proper subvarieties for ever $n\geq 1$.  Indeed, \cite{k3-families} constructs non-isotrivial families $X\to S$ of polarized $K3$ surfaces over a proper base $S$.  Then $\on{Hilb}^n(X/S)$ gives a family of $2n$-dimensional (weak) Calabi-Yau varieties over $S$.  We do not know of non-isotrivial examples of Calabi-Yau varieties over a proper base where the fibers have $h^{2,0}=0$; such an example would be very interesting.
\end{rem}
\begin{cor}\label{modulispaceapplications}
Let $k$ be a field of characteristic zero, and let $X$ be a smooth projective $k$-variety with $\dim(X)\geq 3$.  Let $D\subset X$ be an ample divisor, and let $f: Y\to D$ be a smooth projective morphism of relative dimension $n$ such that:
\begin{enumerate}
\item The fibers of $f$ are geometrically connected curves of genus $g\geq 1$, and $$\dim(D)>3g-3,$$ or
\item The map $f$ exhibits $Y$ as a family of polarized hyperk\"ahler varieties over $D$, and the moduli space parametrizing these varieties has dimension less than $\dim(D)$, or
\item The map $f$ satisfies a local Torelli theorem for $H^1$ and the relevant moduli space of polarized manifolds is smooth of dimension less than $\dim(D)$ (e.g. if $f$ exhibits $Y$ as a polarized Abelian $X$-scheme), or
\item The map $f$ exhibits $Y$ as a family of polarized (weak) Calabi-Yau varieties over $Y$, with $h^{2,0}=h^{1,0}=0$,  and the relevant moduli space of polarized manifolds is smooth of dimension less than $\dim(D)$.
\end{enumerate}
Then $f$ extends to a smooth family of polarized varieties over $X$.
\end{cor}
\begin{proof}
We must verify the hypotheses of Theorem \ref{nefcotangentbundleextension}, namely that the families in question come from a classifying map to a smooth DM stack with nef cotangent bundle, and admitting a finite \'etale cover by a scheme.  In case (1), smoothness is well-known; in cases (2) and (4), smoothness follows from the Bogomolov-Tian-Todorov theorem.  In case (3) smoothness follows from the smoothness of the moduli space of Abelian varieties with a polarization of a fixed degree.  We may deduce that the moduli spaces in question have nef cotangent bundle from Corollary \ref{mgnefcotangentbundle} in case (1), from Theorem \ref{hyperkahlernef} in case (2), from Lemma \ref{torellinef} in case (3), and from Theorem \ref{calabiyaunef} in case (4).

It only remains to show that the moduli schemes in question admit a finite \'etale cover by a scheme.  In all cases, local Torelli theorems give a finite \'etale cover by an algebraic space, by adding level structure.  But the quasi-projectivity results of Viewheg \cite{viehweg} show that these algebraic spaces are schemes.
\end{proof}
\begin{rem}
The examples in Corollary \ref{modulispaceapplications} include most examples of smooth moduli spaces known to the author.  We expect, however, that these results are true in significantly more generality (e.g. even in the case of singular moduli spaces).  Perhaps one reason to believe this is the general philosophy that moduli spaces of polarized varieties should be ``hyperbolic," as exemplified by the work of Viehweg-Zuo \cite{viehweg-zuo, viehweg-zuo2}, Moller-Viehweg-Zuo \cite{moller-viehweg-zuo}, and Kovacs \cite{kovacs}, among others.  

To ask a more precise question:  let $\pi: X\to Y$ be a family of smooth polarized variety, with tangent complex $\mbf{T}_{X/Y}$.  What is the Frobenius amplitude of $(\mbf{R}\pi_*\mbf{T}_{X/Y})^\vee$?  Information on this sort of question would be useful for generalizing the results above to singular moduli spaces.
\end{rem}
We also observe that one may recover the Lefschetz hyperplane theorem for $\pi_1^{\text{\'et}}$, via these techniques.  Recall the statement:
\begin{thm}[Found in {\cite[Th\'eor\`eme 3.10]{SGA2}}]\label{etalepi1-lefschetz}
Let $k$ be a field and $X$ a smooth projective $k$-variety.  Let $D\subset X$ be an ample divisor.  Then the natural map (obtained after choosing a base-point) $\pi_1(D)\to\pi_1(X)$ is
\begin{itemize}
\item surjective if $\dim(X)\geq 2$, and
\item an isomorphism if $\dim(X)\geq 3$.
\end{itemize}
\end{thm}
\begin{proof}
This is immediate from Theorem \ref{charzeroextensionthmdmstacks} if $k$ is of characteristic zero, by applying the theorem in the case that $Y=BG$, for $G$ a finite \'etale group scheme.  In characteristic $p>0$ we may deduce the result from Theorem \ref{frobeniusvanishing}, but we omit the proof.
\end{proof}
\subsection{Variations of Hodge structure and period domains}
Observe that by Theorem \ref{sourceofnefexamples}(4), compact quotients of Hermitian symmetric domains by arithmetic groups (viewed as Deligne-Mumford stacks) have nef cotangent bundle; by results of Baily-Borel \cite{baily-borel}, these (a priori analytic) Deligne-Mumford stacks are in fact algebraic and admit finite \'etale covers by schemes (by increasing level structure).  Likewise, non-compact Shimura varieties of PEL type have nef cotangent bundle, by Lemma \ref{torellinef}, and again admit finite \'etale covers by schemes.  Thus, we have proven
\begin{thm}\label{shimuralefschetz}
Let $Y$ be a compact (stack) quotient of a Hermitian symmetric domain by an arithmetic group, or a Shimura variety of PEL type.  Let $X$ be a smooth projective variety, and $D\subset X$ an ample divisor, with $\dim(Y)<\dim(D)$.  Then any map $D\to Y$ extends uniquely to a map $X\to Y$.
\end{thm}
As a corollary, we see that if $\mcl{H}\to D$ is a polarized variation of Hodge structure whose induced period map is given by a map $D\to Y$, where $Y$ is a compact quotient of a Hermitian symmetric domain  or a Shimura variety of PEL type with $\dim(Y)<\dim(D)$, then $\mcl{H}$ extends uniquely to a polarized variation of Hodge structure on $X$.  We conjecture that this is true in significantly greater generality than we prove it here.  Namely, 
\begin{conj}\label{vhsconjecture}
Let $D$ be an ample divisor in a smooth projective variety $X$.  Let $\mcl{H}\to D$ be a polarized variation of Hodge structure.  Suppose that $\dim(D)$ is large in terms of the numerical invariants of $\mcl{H}$.  Then $\mcl{H}$ extends uniquely to a polarized variation of Hodge structure on $X$.
\end{conj}
Theorem \ref{shimuralefschetz} proves many special cases of this conjecture; we expect a complete proof in the case that $D$ has mild singularities to appear in upcoming work \cite{litt}.  The case where $D$ is smooth follows from \cite[Corollary 4.3]{simpson2}.

Observe that work of Simpson \cite{simpson} gives a purely algebro-geometric interpretation of this conjecture.  Namely, Simpson identifies an algebraic category (polystable Higgs bundles) corresponding to the category of polarized variations of Hodge structure; thus one may give Conjecture \ref{vhsconjecture} a purely algebraic statement.

We believe Conjecture \ref{vhsconjecture} to be of some importance, is it would give a natural way of extending vector bundles from an ample divisor to an ambient variety---namely, vector bundles arising from certain polarized variations of Hodge structure would extend.
\subsection{Maps to varieties with $f$-semipositive cotangent bundle}\label{f-semipositive-section}
We now give some results in arbitrary characteristic.  While in the previous section, we required throughout that $\dim(Y)<\dim(D)$, this section will have no such requirement; we will only require that $\dim(X)\geq 3$.  Our fundamental results will be in the case that $Y$ has f-semipositive cotangent bundle (recall that f-semipositivity was defined in Definition \ref{f-semipositivity-def}).

Our main result will be:
\begin{thm}\label{f-semipositive-lefschetz-thm}
Let $X$ be a smooth projective variety over a field $k$, $D\subset X$ an ample divisor, and set $n:=\dim(X)\geq 2$.  Suppose that either $\on{char}(k)=0$ or that $k$ is perfect of characteristic $p>n$ and $X$ lifts to $W_2(k)$.  Let $Y$ be a smooth proper $k$-variety with $f$-semipositive cotangent bundle.  Then $${\Hom}(X, Y)\to {\Hom}(D, Y)$$ is
\begin{itemize}
\item Injective if $\dim(X)=2$, and
\item Bijective if $\dim(X)\geq 3$.
\end{itemize}
\end{thm}
Before proceeding with the proof, we will need a lemma on the geometry of varieties with nef and f-semipositive cotangent bundle.
\begin{prop}\label{noratlcurves}
Let $X$ be a smooth variety with nef cotangent bundle.  Then $X$ contains no rational curves.
\end{prop}
\begin{proof}
Suppose to the contrary that there is a non-constant morphism $f: \mbb{P}^1\to X$.  Then the image of $f$ is a rational curve, $C$; taking its normalization gives an unramified map $\iota: \mbb{P}^1\to X$.  Thus there is a surjection $$\iota^*\Omega^1_X\to \Omega^1_{\mathbb{P}^1}\to 0.$$ But $\Omega^1_{\mbb{P}^1}$ has negative degree, contradicting the nefness of $\Omega^1_X$.
\end{proof}
\begin{prop}\label{f-semipositive-noratlcurves}
Let $X$ be a smooth projective variety with $f$-semipositive cotangent bundle.  Then $X$ contains no rational curves. 
\end{prop}
\begin{proof}
We use that $f$-semipositive vector bundles are arithmetically nef \cite[Lemma 3.13]{arapura-f-amplitude}.  Then we may conclude using Proposition \ref{noratlcurves}. 
\end{proof}
\begin{proof}[Proof of Theorem \ref{f-semipositive-lefschetz-thm}]
We must verify the hypotheses of Theorems \ref{unique-extensions}, \ref{charzeroextensionthm} and \ref{charpliftextensionthm}.  First, given a map $f: D\to Y$ we will show $$\phi(f^*\Omega^1_Y\otimes \mcl{N}_{D/X})=0,$$ and given a map $g: X\to Y$ we will show $$\phi(g^*\Omega^1_Y(D))=0.$$  But f-semipositivity of vector bundles is preserved by pullback \cite[Proposition 3.10]{arapura-f-amplitude}, so $f^*\Omega^1_Y, g^*\Omega^1_Y$ are both f-semipositive.  Now we may conclude the result by Theorem \ref{tensorproductbound}, using that $\mcl{O}_X(D), \mcl{N}_{D/X}$ are ample line bundles and thus have f-amplitude zero.

We also observe that by Proposition \ref{f-semipositive-noratlcurves}, $Y$ admits a model over a finite type $\mbb{Z}$-scheme whose geometric fibers contain no rational curves, so all the hypotheses are satisfied, as desired.
\end{proof}
As usual, there is an analogue of this theorem for sections to morphisms with f-semipositive relative cotangent bundle.

We now provide some examples of varieties with f-semipositive cotangent bundle.  Before stating our main theorem, we need to recall the notion of an arithmetically nef line bundle.
\begin{defn}
Let $k$ be a field, $X$ a $k$-variety, and $\mcl{L}$ a line bundle on $X$.  If $k$ is of positive characteristic, we say that $\mcl{L}$ is \emph{arithmetically} nef if it is nef.  If $k$ has characteristic zero, we say that $\mcl{L}$ is \emph{arithmetically nef} if there exists a model $(\overline{X}, \overline{\mcl{L}})$ of $(X, \mcl{L})$ over a finite-type $\mbb{Z}$-scheme $S$, so that for all closed $s\in S$, the line bundle $\overline{\mcl{L}}_s$ on $\overline{X}_s$ is nef.
\end{defn}
The main result we will need about arithmetically nef line bundles is:
\begin{prop}[{\cite[Proposition A.2]{arapura-f-amplitude}}]\label{arithmeticallynefiff-f-semipositive}
A line bundle on a projective variety is f-semipositive if and only if it is arithmetically nef.
\end{prop}
\begin{thm}\label{source-of-f-semipositive-examples}
Let $Y$ be a smooth projective variety over a field $k$ such that one of the following holds:
\begin{enumerate}
\item $Y$ is a curve of genus at least one.
\item $Y$ has trivial cotangent bundle.
\item There exists a smooth map $f: Y\to X$ with $\Omega^1_X, \Omega^1_{Y/X}$ f-semipositive.
\item There exists an \'etale morphism $g: Y\to Y'$ with $\Omega^1_{Y'}$ f-semipositive.
\item There exists a finite \'etale morphism $g: Y'\to Y$ with $\Omega^1_{Y'}$ f-semipositive and such that $\on{char}(k)$ does not divide $\on{deg}(g)$.
\item $Y$ is a divisor in smooth variety $Y'$ so that $\Omega^1_{Y'}$ is f-semipositive and $\mcl{N}_{Y/Y'}^\vee$ is arithmetically nef. 
\end{enumerate}
Then $\Omega^1_Y$ is f-semipositive.
\end{thm}
\begin{proof}
In every case it suffices to work in positive characteristic.  
\begin{enumerate}
\item A curve of genus one has arithmetically nef cotangent bundle, so the result is immediate from Proposition \ref{arithmeticallynefiff-f-semipositive}.
\item Recall that we must show that the Castelnuovo-Mumford regularity of $(\Omega^1_Y)^{(p^n)}$ remains bounded as $n$ goes to infinity.  But if $\Omega^1_Y$ is trivial, we have $(\Omega^1_Y)^{(p^n)}=\Omega^1_Y$, so this is clear.
\item This is immediate from the existence of the short exact sequence $$0\to f^*\Omega^1_X\to \Omega^1_Y\to \Omega^1_{Y/X}\to 0,$$ because $f^*\Omega^1_X, \Omega^1_{Y/X}$ are both f-semipositive by assumption (recall that the pullback of an f-semipositive vector bundle is f-semipositive, by \cite[Proposition 3.10]{arapura-f-amplitude}).
\item We have $\Omega^1_Y=g^*\Omega^1_{Y'}$, so $\Omega^1_Y$ is nef by \cite[Proposition 3.10]{arapura-f-amplitude}.
\item There is a natural injective map $\Omega^1_Y\to g_*g^*\Omega^1_Y=g_*\Omega^1_{Y'}$, which is split by $\frac{1}{\on{deg}(g)}\on{tr}_{Y'/Y}$.  Thus it suffices to show that $g_*\Omega^1_{Y'}$ is f-semipositive.  Let $\mcl{O}_Y(1)$ be an ample line bundle on $Y$ and $\mcl{O}_{Y'}(1)$ be its pullback to $Y'$ (which is also ample).   By the projection formula, $$(g_*\Omega^1_{Y'})(n)=g_*(\Omega^1_{Y'}(n)),$$ so $$H^i(Y, (g_*\Omega^1_{Y'})(n))=H^i(Y', \Omega^1_{Y'}(n)).$$  This gives the result.
\item Consider the short exact sequence $$0\to \mcl{N}_{Y/Y'}^\vee\to \Omega^1_{Y'}|_Y\to \Omega^1_Y\to 0.$$ The middle term is f-semipositive by assumption, and the first is f-semipositive as it is an arithmetically nef line bundle (Proposition \ref{arithmeticallynefiff-f-semipositive}).  So the last term is f-semipositive, as desired. 
\end{enumerate}
\end{proof}
\begin{rem}
This theorem allows us to construct many examples of varieties with f-semipositive cotangent bundle.  For example, bi-elliptic surfaces and total spaces of Kodaira fibrations both have f-semipositive cotangent bundle.
\end{rem}
\subsection{Maps to small targets}
We now consider the case of maps $f$ to a smooth target $Y$ with $\dim(\on{im}(f))<\dim(D)-1$.  Our main result is that in this case, a map $D\to Y$ \emph{always} extends to a map $X\to Y$.
\begin{lem}\label{smalltargetbound}
Let $X$ be a finite-type $k$-scheme, and $\mcl{E}$ a vector bundle on $X$.  Let $Y$ be another finite-type $k$-scheme, and $\mcl{F}$ a vector bundle on $Y$.  Suppose $f: X\to Y$ is a morphism.  Then $$\phi(\mcl{E}\otimes f^*\mcl{F})\leq \phi(\mcl{E})+\dim(\on{im}(f)).$$
\end{lem}
\begin{proof}
It suffices to work in positive characteristic.  Let $\mcl{G}$ be a coherent sheaf on $X$.  Then 
\begin{align*}
\mbf{R}f_*(\mcl{G}\otimes \on{Frob}_{p}^{k*}(\mcl{E}\otimes f^*\mcl{F}))& =\mbf{R}f_*(\mcl{G}\otimes \on{Frob}_{p}^{k*}(\mcl{E})\otimes f^*\on{Frob}_{p}^{k*}\mcl{F})\\
&=\mbf{R}f_*(\mcl{G}\otimes \on{Frob}_{p}^{k*}(\mcl{E}))\otimes \on{Frob}_{p}^{k*}\mcl{F}
\end{align*} 
by the projection formula.  For $k\gg 0$, $\mbf{R}f_*(\mcl{G}\otimes \on{Frob}_{p}^{k*}(\mcl{E}))$ is concentrated in degrees $[0, \phi(\mcl{E})]$; thus by dimensional vanishing, $$\mbf{R}\Gamma(\mcl{G}\otimes \on{Frob}_{p}^{k*}(\mcl{E}\otimes f^*\mcl{F}))= \mbf{R}\Gamma(\mbf{R}f_*(\mcl{G}\otimes \on{Frob}_{p}^{k*}(\mcl{E}))\otimes \on{Frob}_{p}^{k*}\mcl{F})$$ is concentrated in degrees $[0, \phi(\mcl{E})+\dim(\on{im}(f))]$ for $k\gg0$ as desired.
\end{proof}
As a corollary, we deduce a general Lefschetz theorem for maps to small targets.
\begin{thm}\label{smalltargetsextension}
Let $k$ be a field of characteristic zero.  Let $X$ be a smooth projective $k$-variety with $\dim(X)\geq 3$, and let $D\subset X$ be an ample Cartier divisor.  Let $Y$ be a smooth Deligne-Mumford stack over $k$ with quasi-projective coarse moduli space, and let $f: D\to Y$ be a morphism.  Suppose that  $$\dim(\on{im}(f))<\dim(D)-1.$$  Then $f$ extends uniquely to a morphism $X\to Y$. 
\end{thm}
\begin{proof}
We must verify the hypotheses of Theorems \ref{charzeroextensionthmdmstacks}.  The only non-obvious hypothesis is that $$\phi(f^*\Omega^1_Y\otimes \mcl{N}_{D/X})<\dim(D)-1.$$  But $\phi(\mcl{N}_{D/X})=0$ as $\mcl{N}_{D/X}$ is an ample line bundle, so this is immediate from Lemma \ref{smalltargetbound}.
\end{proof}
\begin{rem}
We observe that the dimension estimates in Theorems \ref{nefcotangentbundleextension} and \ref{smalltargetsextension} are sharp.  Let $X=\mbb{P}^3$ and $D\subset X$ a smooth quadric surface.  Then any non-constant $D\to \mbb{P}^1$ (of which there are many) fails to extend to a map $X\to \mbb{P}^1$, showing that the nefness condition on $\Omega^1_Y$ in Theorem \ref{nefcotangentbundleextension} and the dimension condition in Theorem \ref{smalltargetsextension} cannot be weakened.
\end{rem}
\begin{rem}
As usual, we may also prove a similar theorem for varieties over a perfect field $L$ of positive charactersitic which lift to $W_2(L)$. 
\end{rem}
\subsection{A non-Abelian Noether-Lefschetz theorem for Abelian schemes}
We now come to an application of these ideas to \emph{generic, sufficiently ample} divisors.  The model theorem to consider in this case is the Noether-Lefschetz theorem, which states that the Picard group of a smooth threefold $X$ is the same as that of a very general, sufficiently ample divisor $D\subset X$.  We prove a version of this result for sections to Abelian schemes.

First, we observe that there is an existing result along these lines.  Namely, Fakhruddin \cite[Proposition 4.1]{fakhruddin} proves an analogue of the Noether-Lefschetz theorem for sections to Abelian schemes over general, sufficiently ample divisors.  We give a similar but slightly different result; namely we require the divisor to be sufficiently ample, but it need not be general.
\begin{thm}\label{noether-lefschetz-abelian-schemes}
Let $k$ be a field, and $X$ a smooth projective $k$-scheme of dimension $m\geq 3$.  Let $f: A\to X$ be an Abelian scheme, and let $\mcl{O}_X(1)$ be an ample line bundle on $X$.  Then for $n\gg 0$ (depending  on $A, X, \mcl{L}$), and any element $D\in |\mcl{L}^{\otimes n}|$, the restriction map $$\on{Sections}(f)\to \on{Sections}(f_D)$$ is an isomorphism.
\end{thm}
\begin{proof}
Let $s: X\to A$ be the identity section to $f$.  We choose $n\gg0$ so that $$H^{i}(X, s^*\Omega^1_{A/X}(n')\otimes \omega_X)=0$$ for $i=m, m-1, m-2$ and all $n'\geq n$; such an $n$ exists by Serre vanishing.

Let $D\in |\mcl{L}^{\otimes n}|$ and let $r: D\to A$ be a section to $f_D$; we wish to show that $r$ extends uniquely to $X$.  Now by Serre duality and the short exact sequence $$0\to (s^*\Omega^1_{A/X})^\vee(-n'-1)\to (s^*\Omega^1_{A/X})^\vee(-n')\to (s^*\Omega^1_{A/X})^\vee(-n')|_D\to 0$$ we have that \begin{equation}\label{abelianschemevanishing}H^0(D, (s^*\Omega^1_{A/X})^\vee(-n')|_D)=H^1(D, (s^*\Omega^1_{A/X})^\vee(-n')|_D)=0\end{equation} for all $n'\geq n$.  But observe that $$s^*\Omega^1_{A/X}=f_*\Omega^1_{A/X}=r^*\Omega^1_{A/X}.$$  Thus Equation \ref{abelianschemevanishing} shows that the obstructions and deformations to extending $r$ to a section on the infinitesimal neighborhoods of $D$ vanish.  So $r$ extends uniquely to a section $\widehat D\to A$ to $f_{\widehat D}$, where $\widehat D$ is the formal scheme obtained by completing $X$ at $D$.

But by Corollary \ref{sectionalgebraization}, this map automatically extends to a section on some open neighborhood $U$ of $D$; as Abelian varieties contain no rational curves, such a section extends to all of $X$ by Proposition \ref{noratlcurvesextension}.  Such an extension is unique, by Corollary \ref{sectionsareequal}.
\end{proof}
\begin{rem}
By Corollary \ref{sectionswithglobalgeneration}, we may take $n=1$ above if $$\on{rel. dim.}(f)<\dim(X)-1.$$
\end{rem}

\hypersetup{linkcolor=blue}
\bibliographystyle{alpha}
\bibliography{thesis-bibtex}


\end{document}